\documentclass[onefignum,onetabnum]{siamart190516}

\usepackage[utf8]{inputenc}
\usepackage[english]{babel}

\usepackage{amsmath, amssymb,mathtools}
\usepackage{dsfont}
\usepackage{cancel}
\usepackage{xfrac}		
\usepackage{mathdots}	
\usepackage{thmtools} 	
\usepackage{tcolorbox}	

\makeatletter
\renewcommand*\env@matrix[1][\arraystretch]{%
  \edef\arraystretch{#1}%
  \hskip -\arraycolsep
  \let\@ifnextchar\new@ifnextchar
  \array{*\c@MaxMatrixCols c}}
\makeatother

\usepackage{algorithmic}

\usepackage{graphicx}
\usepackage{subcaption}
\usepackage{caption}

\usepackage{pgfplots}	

\usepackage{enumerate}
\usepackage{paralist}

\usepackage{epstopdf}
\usepackage{graphicx}
\usepackage{subcaption}
\usepackage{float}
\usepackage[export]{adjustbox} 

\usepackage{afterpage}	
\usepackage{pgfplots}	

\usepackage{multirow}

\usepackage{tikz}
\usetikzlibrary{calc}
\usetikzlibrary{decorations.markings}
\usetikzlibrary{shapes.geometric,positioning}
\usetikzlibrary{arrows}
\tikzset{vertex/.style={minimum size=2mm,circle,fill=black,draw,inner sep=0pt},
         decoration={markings,mark=at position .5 with {\arrow[black,thick]{stealth}}}}
\pgfplotsset{compat=1.10}
\usepgfplotslibrary{fillbetween}
\usetikzlibrary{backgrounds}
\usetikzlibrary{patterns}         

\usepackage{color}
\definecolor{vertFonce}{rgb}{0,0.5,0}
\definecolor{numLignes}{rgb}{0.17,0.57,0.7} 
\definecolor{gris}{rgb}{0.5,0.5,0.5}
\definecolor{grisFonce}{rgb}{0.2,0.2,0.2}
\definecolor{orange}{rgb}{1,0.65,0.31}  
\definecolor{orangeFonce}{rgb}{1,0.4,0}
\definecolor{bleuFonce}{rgb}{0,0,0.4}
\definecolor{rougeFonce}{rgb}{0.3,0,0}
\definecolor{rougeWord}{rgb}{0.5,0,0}
\definecolor{vertClair}{rgb}{0.8,1,0.8}
\definecolor{rougeClair}{rgb}{1,0.5,0.5}    
\newcommand{\R}{ \mathbb{R} }
\newcommand{\N}{ \mathbb{N} }
\newcommand{\C}{ \mathbb{C} }

\renewcommand{\Re}[1]{\operatorname{Re}\{#1\}}

\newcommand{\norm}[1]{\left\lVert #1 \right\rVert}

\newcommand{\e}[1]{\mathrm{e}^{#1}}

\newcommand*\tnlm{t_n^{\ell-}}
\newcommand*\tnlp{t_n^{\ell+}}
\newcommand*\J{{\cal J}}

\renewcommand{\t}{\top} 

\newcommand{\F}{ \mathcal{F} }
\newcommand{\G}{ \mathcal{G} }

\newcommand{\Min}[1]{\raisebox{0.5ex}{\scalebox{0.8}{$\displaystyle \min_{#1}\;$}}}

\newtheorem{remark}{Remark}

\newcommand{\FK}[1]{{\color[rgb]{0,0,0}#1}}
\newcommand{\JS}[1]{{\color[rgb]{0,0,0}#1}}
\newcommand{\SR}[1]{{\color[rgb]{0,0,0}#1}}

\headers{Time-parallelization of sequential observers}{F. Kwok, S. Riffo, and J. Salomon}

\title{Time-parallelization of sequential data assimilation problems}

\author{Sebasti\'an RIFFO\thanks{CEREMADE, CNRS, UMR 7534, Université Paris-Dauphine, PSL Research University, 75016 Paris, France
  (\email{reyesriffo@ceremade.dauphine.fr}).}
\and Felix Kwok\thanks{D\'epartement de math\'ematiques et de statistique, Universit\'e Laval
  (\email{felix.kwok@mat.ulaval.ca}).}
\and Julien Salomon\thanks{INRIA  Paris,  ANGE  Project-Team,  75589  Paris  Cedex  12,  France
   and Sorbonne Universit\'e, CNRS, Laboratoire Jacques-Louis Lions, 75005 Paris, France (\email{julien.salomon@inria.fr})}.
}

\begin{document}

\maketitle

\begin{abstract}
This paper is devoted to the problem of time parallelization of assimilation methods applying on unbounded time domain. 
In this way, we present a general procedure to couple the Luenberger observer with time parallelization algorithm.
Our approach is based on \textit{a posteriori} error estimates of the latter and preserves the rate of the non-parallelized observer.
We then focus on the case where the Parareal algorithm is used as time parallelization algorithm, and derive a bound of the efficiency of our procedure.
A variant devoted to the case a large number of processors is also proposed. We illustrate
the performance of our approach with numerical experiments. 
\end{abstract}



\section{Introduction}
The assumptions behind a mathematical model not only determine their range of applicability, but also induce an inevitable gap between predictions and reality. In order to narrow this difference, one can sacrifice the simplicity of the model or incorporate real data instead, by following a data assimilation (DA) procedure. Among these approaches, \textit{sequential} methods construct a new system which uses the available observations (that arrive uninterrupted in time) to approximate the true state, whereas \textit{variational} methods follow an optimal control approach using the information collected in a fixed amount of time.

In a deterministic context, sequential methods are often called \textit{observers}. In the pioneering work~\cite{L}, Luenberger introduced a dynamic which imitates the original model, by including an extra term that measures the misfit between the observations and its own predictions. As long as the original model is \textit{observable}, 
this error 
can be driven to zero at exponential rate by properly choosing a certain matrix, 
meaning that the true state is recovered asymptotically.
An alternative is the \textit{Kalman filter}~\cite{K}, which takes into account measurement errors and model inaccuracies represented by Gaussian white noises (both stationary and mutually uncorrelated), in order to compute a state estimate that minimizes the mean square error. Note that extensions to the nonlinear case have been developed, e.g., nonlinear versions of Luenberger observer~\cite{NLL} and  \textit{Extended Kalman filter}~\cite{Notation_DA}. 

Significant difficulties appear when applying these techniques to, e.g., meteorology~\cite{Hoke,Auroux-00} or oceanography problems~\cite{Lyne,verron,lorenc,AUROUX-09-SW}. 
Here,  the number of state variables and the vast amount of observations lead to very costly computations. To overcome this issue, one can consider
 space or time domain decomposition methods, which accelerate the numerical solution of PDEs using parallel computing. We now briefly recall the main approaches and refer to~\cite{Gander-SchwarzMethods,G50} for their detailed description. 

%

 Since the seminal work of Schwarz~\cite{Schwarz}, spatial domain decomposition 
 and corresponding parallelization techniques have seen many qualitative and quantitative improvements, all of which are nowadays well 
 documented~\cite{Gander-SchwarzMethods,MR3450068}. The time direction is significantly more complex to parallelize.  The solution process of evolutionary PDEs is indeed intrinsically sequential, so that time decomposition is, at first glance, not amenable to parallel computing. However, over the last 50 years, many parallel-in-time methods have been developed~\cite{G50}. 
 The  origins of these approaches can be traced back to Nievergelt~\cite{Nievergelt}, who first introduced the concept that has later been dubbed~\textit{Multiple shooting}: decompose the time interval into disjoint subintervals and solve simultaneously a family of initial-value problems, breaking the intrinsic sequential nature of the time-dependent differential equation. Among these methods, one of the most recent ones is the Parareal algorithm~\cite{LMT}. 

Different procedures have been developed to couple space or time parallel methods with data assimilation problems. Trémolet and Le Dimet~\cite{4DVar-par} were among the first to address the parallelization of \textit{Variational} data assimilation problems in meteorology. In a continuous setting, they proposed a domain decomposition approach combined with an adjoint method, by assigning to each subdomain a local version of a continuous minimization problem, with an extra term on the local cost functional to enforce the continuity of the state between adjacent domains. 
Following this approach, Rao and Sandu~\cite{RS} apply a quasi-Newton solver to the \FK{4D-Var problem~\cite{4DVAR}}, and time-parallelize the computation of the gradient. A more sophisticated approach is proposed by D'Amore and Cacciapuoti~\cite{Amore-Cacciapuoti}, who combine the Parareal algorithm with the Multiplicative Parallel Schwarz method (MPS) to solve 4D-Var. Note finally that time paralellization has also been combined \FK{with} optimization solvers in the neighbor field of control~\cite{MST,GKS}. 

Parallel-in-time algorithms could be quite useful when dealing with \FK{long} 
time intervals, as is the case of \textit{sequential} DA methods, where information can arrive uninterrupted. 
However, coupling these two approaches is not straightforward since the former generally applies on bounded time intervals. In this paper, we propose a first general method to \FK{time-parallelize} an unbounded assimilation method, namely, the Luenberger observer. Our approach is based on a sequential treatment of time windows, each windows being processed in parallel. 

Our paper is organized as follows: we start in Section \ref{sec:diamond} by proposing a procedure to couple sequential data assimilation methods with parallel-in-time algorithms, \FK{which involves} 
splitting the unbounded time interval into subintervals of the same length (\textit{windows}) and then apply\FK{ing}, following a sequential order, the time-parallel solver on each \FK{window}. By considering the Luenberger observer as \FK{an} assimilation method, we provide an accuracy criterion that preserves its exponential rate of convergence, which yields an \textit{a posteriori} estimate of the accuracy of the solver. In order to go further, in Section \ref{sec:parallel} we \FK{use} 
the Parareal algorithm as \FK{a} parallel-in-time solver. This allows us to design an alternative algorithm that provides an \textit{a priori} estimate of the number of iterations required on each window, which also enables us to investigate the theoretical efficiency of the entire procedure. These results are based on a new convergence estimate that we derive for Parareal when the coarse solver is a contraction mapping. Finally, we present some numerical results in Section \ref{sec:numerics}.

Throughout this paper, $\norm{\cdot}$ represents the induced 2-norm of a matrix.


\section{The Luenberger observer}\label{sec:Luenberger}
Control theory usually requires a complete knowledge of the state vector. However, due to certain limitations related to a problem, for instance the number of available measurements, one can often have access only to partial information. An example which fits into this setting is given by 
\begin{equation}\left\{
\begin{aligned}
	\dot{x}(t) =&  Ax(t) +Bu(t),\quad x(0) =  x_0,\\
	y(t) =&  Cx(t),
\end{aligned}\right. \label{det-assim}
\end{equation}
where $A\in\mathcal{M}_{m\times m}(\R)$, $B\in\mathcal{M}_{m\times p}(\R)$ and $C\in\mathcal{M}_{q\times m}(\R)$ are assumed to be known. Here $m, p, q \in \N^{*}$, with $p,q < m$ (and generally $q\ll m$ in the applications);  $x\in\R^m$ is the state vector, $y\in\R^q$ is the measured output, $u\in\R^p$ 
and $t\in(0,+\infty)$. The initial condition $x(0) = x_0$ is unknown. 

In such a situation, one can try to compute an estimate $\hat{x}(t)$ of $x(t)$, knowing only the input $u(t)$ and output $y(t)$. To tackle this problem, Luenberger \cite{L} proposed to consider the dynamical system
\begin{equation}\left\{
\begin{aligned}
	\dot{\hat{x}}(t) =&  A\hat{x}(t) +Bu(t) +L\left[y(t)-\hat{y}(t)\right], \quad \hat{x}(0) = \hat{x}_0 \\
	\hat{y}(t) =&  C\hat{x}(t)
\end{aligned}\right. \label{Luenberger}
\end{equation}
with $L\in\mathcal{M}_{m\times q}(\R)$ the \textit{observer gain} and $\hat{x}_0$ an arbitrary initial condition. Equations \eqref{Luenberger} are known as the \textit{Luenberger observer} or the \textit{Identity observer}.

The matrix $L$ needs to be specified, but let us already note that it plays an important role in the estimation error $\varepsilon(t):=x(t)-\hat{x}(t)$. Indeed, substracting \eqref{det-assim} and \eqref{Luenberger}, and then solving the resultant ODE, one obtains 
\begin{equation}
	\varepsilon(t) = \e{(A-LC)t}\left(x(0)-\hat{x}(0)\right). \label{error}
\end{equation} 
This last equality shows that the error will decay to zero if the eigenvalues of $A-LC$ lie in the open left half-plane $\{z\in\C: \, \Re{z} < 0\}$, where $\Re{z}$ denotes the real part of $z$. This property is related to \FK{the} observability condition. More precisely, recall that \eqref{det-assim} is \textit{observable} if the rank of the matrix 
\[
\mathcal{C} \vcentcolon= \begin{bmatrix} C \\ CA \\ \vdots \\ CA^{m-1}\end{bmatrix}
\]
is $m$. 
We then have the following result, often calle\FK{d} the \textit{Identity observer Theorem }\cite[p.303]{L}\FK{:}
\begin{theorem}\label{observability}
Given a completely observable system \eqref{det-assim}, an identity observer of the form \eqref{Luenberger} exists. Moreover, the eigenvalues of $A-LC$ can be selected arbitrarily.
\end{theorem}
This theorem shows that given a set $\{\mu_i\}_{i=1,\dots, m}$, 
there exists $L$ satisfying
\begin{equation}
	\det \left(sI -(A-LC)\right) = \phi(s), \label{detL}
\end{equation}
where $\phi(s) =  (s-\mu_1)\cdots (s-\mu_m)$, i.e., $\{\mu_i\}_{i=1,\dots m}$ are the eigenvalues of $A-LC$.

Note that for a single-input single-output system, i.e. $p=q=1$, one could determine a unique $L\in \R^m$ by equating the $m$ coefficients of both polynomials in \eqref{detL}. However, this approach leads to highly nonlinear equations that are in practice not tractable. Another way to proceed is the Bass-Gura method \cite{BG}, which requires the first companion form of $A$ and the coefficients of $\phi(s)$. An even more direct method is the Ackermann's formula \cite{AckSISO} for an observable system, given by 
\begin{equation*}
	L = \phi(A) \mathcal{C}^{-1} (0 \cdots 0\;\; 1)^{\t},
\end{equation*}
which follows from the Cayley-Hamilton Theorem. For its multi-input multi-output extension, see \cite{AckMIMO}.

Due to Theorem \ref{observability},  we obtain
\begin{proposition} \label{stability}
Suppose that \eqref{det-assim} is observable and that the eigenvalues of $A-LC$ are negative and simple. Then, we have 
\begin{equation*} 
	\norm{\e{(A-LC)t}} \leq \gamma \e{-\mu t},
\end{equation*}
with $\mu \vcentcolon = \Min{\nu \in \sigma(A-LC)}{|\nu|}$ and $\gamma \vcentcolon = \textrm{cond}(V) = \norm{V^{-1}}\norm{V}$, where $V$ is the matrix whose rows are the eigenvectors of $A-LC$.
\end{proposition}
Combining the latter with~\eqref{error}, we obtain in particular
\begin{equation}
	\norm{\varepsilon(t)} \leq \gamma \norm{x(0) -\hat{x}(0)} \e{-\mu t}. \label{luenberger-decay}
\end{equation}
In practice, the term $\norm{x(0) -\hat{x}(0)}$ is unknown, whereas $\mu$ is chosen 
 by the procedure that designs $L$, hence known explicitly. Consequently, the previous formula provides in practice only a rate of convergence for the Luenberger observer.


\section{Time-parallelization setting}\label{sec:diamond}

In what follows, we propose to extend \SR{the combination between data assimilation algorithms and parallelization procedures }
to unbounded time intervals\SR{, by considering the Luenberger observer. In this case, }
we will manage to preserve the exponential rate of convergence of the problem, by an approach that we call the \textit{Diamond strategy}. 

Let us \FK{briefly describe} our approach. We proceed by partitioning $[0,+\infty)$ into intervals of the same length that we call \textit{windows}. Following a sequential order, we apply a parallel-in-time solver in each of them, up to some level of accuracy related to a specific accuracy criterion. We then develop an analysis which decomposes the estimation error into two terms, corresponding respectively to the Luenberger observer and the parallelization error. Based on that, we propose a suitable bound on the latter, so that our criterion preserves Luenberger’s rate of convergence.

\subsection{Framework}\label{parallel-setting} 
In order to accelerate the assimilation and take advantage of a time-parallelization procedure, we propose to divide the time interval $[0,+\infty)$ into windows of a given length $T>0$ denoted by 
\begin{equation*}
	W_{\ell} \vcentcolon = (T_{\ell-1},T_{\ell}),\quad \ell \geq 1,
\end{equation*}
where $T_{\ell} = \ell \cdot T$ with $\ell \in\N$. Then, we solve  \eqref{Luenberger} on each window, in a sequential order, using a time-parallel algorithm. Let us describe how this last method applies.

 Given $\ell\geq 1$ and a fixed window $W_\ell$, we decompose the latter into $N$ subintervals of length $\Delta T$
\begin{equation*}
	W_{\ell} = \bigcup_{n=0}^{N-1} (t_n^{\ell},t_{n+1}^{\ell}),
\end{equation*}
with $t_n^{\ell} = T_{\ell-1} +n\Delta T$ and $N\Delta T=T$, \SR{as shown in Figure \ref{fig:notation}.}

\begin{center}
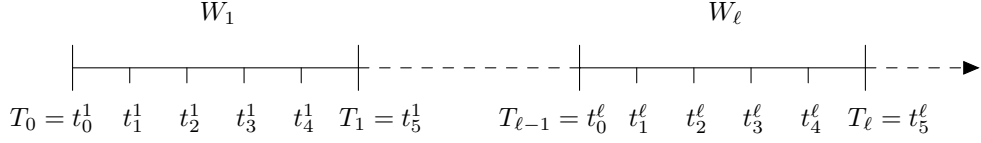
\begin{figure}[h]
\begin{tikzpicture}[line cap=round,line join=round,>=triangle 45,x=.95cm,y=1.0cm]
\begin{scope}[scale=.2,xshift=-50 cm]
\draw (2,18)-- (22,18);
\draw [dash pattern=on 4pt off 4pt] (22,18)-- (48-10,18);
\draw (2,19.5)-- (2,16.5);
\draw (22,19.5)-- (22,16.5);
\draw (6,18)-- (6,17);
\draw (10,18)-- (10,17);
\draw (14,18)-- (14,17);
\draw (18,18)-- (18,17);
\draw[color=black] (12.69-.5,22.62-1) node {$W_1$};
\draw[color=black] (2.55-.5-.5-1,14.63) node {$T_0=t^1_0$};
\draw[color=black] (6.76-.5,14.63) node {$t^1_1$};
\draw[color=black] (10.73-.5,14.63) node {$t^1_2$};
\draw[color=black] (14.7-.5,14.63) node {$t^1_3$};
\draw[color=black] (18.71-.5,14.63) node {$t^1_4$};
\draw[color=black] (22.58+1,14.63) node {$T_1=t^1_5$};

\begin{scope}[xshift=-10 cm]
\draw (48,18)-- (68,18);
\draw (48,19.5)-- (48,16.5);
\draw (68,19.5)-- (68,16.5);
\draw [->,dash pattern=on 4pt off 4pt] (68,18)-- (76,18);
\draw[color=black] (58.72-.5,22.62-1) node {$W_\ell$};
\draw[color=black] (48.68-.5-.5-1.5,14.63) node {$T_{\ell-1}=t^\ell_0$};
\draw[color=black] (52.69-.5,14.63) node {$t^\ell_1$};
\draw[color=black] (56.71-.5,14.63) node {$t^\ell_2$};
\draw[color=black] (60.68-.5,14.63) node {$t^\ell_3$};
\draw[color=black] (64.74-.5,14.63) node {$t^\ell_4$};
\draw[color=black] (68.66+1,14.63) node {$T_{\ell}=t^\ell_5$};
\draw (52,18)-- (52,17);
\draw (56,18)-- (56,17);
\draw (60,18)-- (60,17);
\draw (64,18)-- (64,17);
\end{scope}

\end{scope}
\end{tikzpicture}
\caption{Notation associated with the parallelization setting in the case $N=5$.}\label{fig:notation}
\end{figure}
\end{center}

Since time moves forward, parallelizing in this direction requires  on each subinterval the introduction of initial conditions $\hat{X}_{\ell,n}^{h}$. 
 These are assumed to be obtained from the time-parallelization procedure under consideration. In this setting, the parameter $h$ is used in the notation to account  for the accuracy of the procedure. 
In this way, we introduce a parallel version of~\eqref{Luenberger} in each subinterval $(t_n^{\ell},t_{n+1}^{\ell})$, namely
\begin{equation}\left\{
\begin{aligned}
	\dot{\hat{x}}_{\parallel}(t) =& A\hat{x}_{\parallel}(t)+Bu(t) + L\left[y(t)-C\hat{x}_{\parallel}(t)\right] \\
	\hat{x}_{\parallel}(\tnlp) =&  \hat{X}_{\ell,n}^{h}, 
\end{aligned}\right. \label{parallel-continuous}
\end{equation}
where $\hat{x}_{\parallel}(t)$ denotes the approximation of $\hat{x}(t)$ obtained by the parallel-in-time solver, \SR{see Figure \ref{fig:diamond-strategy}.} When $n = 0$, we consider as initial conditions $\hat{X}_{0,0}^h = \hat{x}_0$ and $\hat{X}_{\ell,0}^h = \hat{x}_{\parallel}(T_{\ell}^{-})$ for $\ell > 0$, meaning that $\hat{x}_{\parallel}$ is continuous at $T_\ell$ and that windows are treated sequentially, as announced above. 


\begin{center}
\begin{figure}[h]
\begin{tikzpicture}
	\draw [->] (0,0) -- (10.5,0) node [right] {$t$};
	\draw [->] (0,0) -- (0,3.5); 
	
	\draw [orangeFonce,domain=0:10.5,samples=100] plot (\x, {3 +0.25*(sin(3*\x r)});
	\node[left] at (0,3) {$x_0$};
	
	\draw [blue,domain=0:10.5,samples=100] plot (\x, {3*(1-exp(-(2/3)*\x)) +0.25*(sin(3*\x r)});
	\node[left] at (0,0) {$\hat{x}_0$};
	
	\draw[thick] (2.5,-0.2)-- (2.5,0.2);
	\draw[thick] (5,-0.2)-- (5,0.2);
	\draw[thick] (7.5,-0.2)-- (7.5,0.2);
	\draw[thick] (10,-0.2)-- (10,0.2);

	\draw[gray] (0.5,-0.125)-- (0.5,0.125);
	\draw[gray] (1,-0.125)-- (1,0.125);
	\draw[gray] (1.5,-0.125)-- (1.5,0.125);
	\draw[gray] (2,-0.125)-- (2,0.125);
	\draw[gray] (3,-0.125)-- (3,0.125);
	\draw[gray] (3.5,-0.125)-- (3.5,0.125);
	\draw[gray] (4,-0.125)-- (4,0.125);
	\draw[gray] (4.5,-0.125)-- (4.5,0.125);
	\draw[gray] (5.5,-0.125)-- (5.5,0.125);
	\draw[gray] (6,-0.125)-- (6,0.125);
	\draw[gray] (6.5,-0.125)-- (6.5,0.125);
	\draw[gray] (7,-0.125)-- (7,0.125);
	\draw[gray] (8,-0.125)-- (8,0.125);
	\draw[gray] (8.5,-0.125)-- (8.5,0.125);
	\draw[gray] (9,-0.125)-- (9,0.125);
	\draw[gray] (9.5,-0.125)-- (9.5,0.125);		

	\draw[<->,>=latex] (5,0.75)-- (7.5,0.75);
	\draw (6.25,0.75)  node[above] {$T$};
	\draw[<->,>=latex] (6.5,0.25)-- (7,0.25);
	\draw (6.75,0.25)  node[above] {\tiny{$\Delta T$}};	

	\draw (2.5,-0.1)  node[below] {\small{$T_{\ell-1}$}};
	\draw (5,-0.1)  node[below] {\small{$T_{\ell}$}};
	\draw (7.5,-0.1)  node[below] {\small{$T_{\ell+1}$}};
		
	\draw [vertFonce,domain=0:0.5] plot (\x+2.5, {2 +1.5*(1-exp(-0.5*\x)) +0.25*(sin(3*\x r)});
	\draw [vertFonce,domain=0.5:1] plot (\x+2.5, {1.5 +1.5*(1.1-exp(-0.5*\x)) +0.25*(sin(3*\x r)});
	\draw [vertFonce,domain=1:1.5] plot (\x+2.5, {0.875 +1.5*(1.425-exp(-0.5*\x)) +0.25*(sin(3*\x r)});
	\draw [vertFonce,domain=1.5:2] plot (\x+2.5, {0.19 +1.5*(1.6-exp(-0.5*\x)) +0.25*(sin(3*\x r)});
	\draw [vertFonce,domain=2:2.5] plot (\x+2.5, {0.15 +1.5*(1.4-exp(-0.5*\x)) +0.25*(sin(3*\x r)});

	\draw [fill, magenta] (2.5,2) circle [radius=0.05];
	\draw [fill, magenta] (3,2.25) circle [radius=0.05];
	\draw [fill, magenta] (3.5,2.125) circle [radius=0.05];
	\draw [fill, magenta] (4,1.625) circle [radius=0.05];
	\draw [fill, magenta] (4.5,1.65) circle [radius=0.05];
	
	\draw[dotted] (4,0) -- (4,3.5);
	\draw (4,0) node[below] {\small{$t_n^{\ell}$}};	
	\draw (4,1.625) node[left] {$\scriptstyle{\color{magenta} \hat{X}_{\ell,n}^h}$};
	\draw (3.9,2.2) node[right] {$\scriptstyle{\color{vertFonce} \hat{x}_{\parallel}({t_n^{\ell}}^{\scalebox{0.75}[1.0]{-}}) }$};

%
	
\end{tikzpicture}
\caption{A time-parallelized observer}
\label{fig:diamond-strategy}
\end{figure}
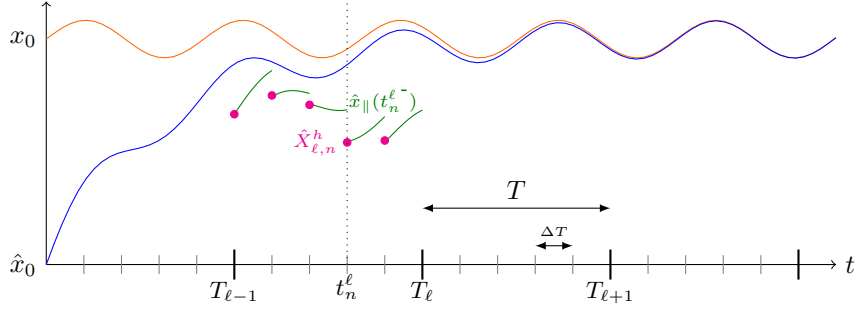
\end{center}

\subsection{The Diamond strategy} \label{diamond} 
Let $\ell\geq 1$. Imposing initial conditions induces discontinuities at the interfaces $t_n^{\ell},\ n=1,\ldots,N-1$ of the subintervals. These \textit{jumps} are 
 defined by $J_{\ell,n}^{h} := \hat{X}_{\ell,n}^h -\hat{x}_{\parallel}(\tnlm)$. 
 The success of the parallel method relies on their decay to zero \FK{as} 
 $\ell$ increases. To analyze this \FK{decay} 
 we clarify the relation between the solution of \eqref{det-assim} and the parallelized observer \eqref{parallel-continuous}.
 
\begin{lemma}\label{parallel-error}
Let $\ell>0$. Under the assumptions of Proposition \ref{stability}, we have
\begin{equation}
\norm{\varepsilon_{\parallel}(T_\ell)}
	\leq \gamma \left( \norm{ x(0)-\hat{x}(0)} +\sum_{j=1}^{\ell} \e{\mu j T}\| \J^h_j\|\right)\e{-\mu\ell T} \label{ineq:parallel-error}
\end{equation}
where $\varepsilon_{\parallel}(t):= x(t)-\hat{x}_{\parallel}(t)$ is the error associated with the approximation  \eqref{parallel-continuous} and
\begin{equation}\label{capJ}
\J^h_\ell:=-\sum_{n=1}^{N-1} \e{(A-LC)(N-n)\Delta T}J_{\ell,n}^{h}.
\end{equation}
\end{lemma}

\begin{proof}
Let $\ell \geq 1$. We have
\begin{align*}
\varepsilon_{\parallel}(\tnlm)
    =&  x(t^\ell_n)-\hat{x}_{\parallel}(\tnlm)
    = \e{(A-LC)\Delta T} \e{-(A-LC)\Delta T}\left(x(t^\ell_n)-\hat{x}_{\parallel}(\tnlm)\right)\\
    =&  \e{(A-LC)\Delta T}(x(t^\ell_{n-1})-\hat{x}_{\parallel}(t^{\ell +}_{n-1}))\\
    =&  \e{(A-LC)\Delta T}(x(t^\ell_{n-1})-\hat{X}^h_{\ell,n-1}))\\
    =& \e{(A-LC)\Delta T}(x(t^\ell_{n-1})-\hat{x}_{\parallel}(t^{\ell -}_{n-1})-J_{\ell,n-1}^{h})\\
    =& \e{(A-LC)\Delta T}(\varepsilon_{\parallel}(t^{\ell -}_{n-1})-J_{\ell,n-1}^{h}  ),
\end{align*}
so that
\begin{equation*}
\varepsilon_{\parallel}(T_{\ell}) =\varepsilon_{\parallel}(t^{\ell}_{N})
= \e{(A-LC)T}\varepsilon_{\parallel}(t^{\ell}_{0})+\J^h_\ell= \e{(A-LC)T}\varepsilon_{\parallel}(T_{\ell-1})+\J^h_\ell,
\end{equation*}
where we have used the continuity of $\hat{x}_{\parallel}$ in $T_\ell=t_N^\ell$ and $T_{\ell-1}=t_N^{\ell-1}$.
In the same way, we obtain
\[\varepsilon_{\parallel}(T_{\ell})=\e{(A-LC) \ell T}\varepsilon_{\parallel}(0) + \sum_{j=1}^{\ell} \e{(A-LC) (\ell-j) T} \J^h_j.\]
The result is obtained by taking the norm and using Proposition \ref{stability}.
\end{proof}
Recall that our approach aims at preserving Luenberger's rate of convergence. Thanks to Lemma~\ref{parallel-error}, we can now define a criterion to reach this goal. 
\begin{proposition}\label{bound-jumps}
Given an arbitrary parameter $\widetilde{\gamma}>0$, let us assume that $h$ satisfies
\begin{equation}
	\gamma\sum_{n=1}^{N-1} \e{\mu n\Delta T}\|J_{\ell,n}^{h}\| \leq \widetilde{\gamma} \frac{\e{-\mu (\ell-1) T }}{2^\ell}, \label{stop-crit}
\end{equation}
for all $\ell\geq 0$. Then, the rate of convergence of $\hat{x}_{\parallel}(t)$ to $x(t)$ is bounded by $\mu$, i.e.
\begin{equation}\label{Luencv}
\norm{\varepsilon_{\parallel}( T_{\ell})} \leq \gamma \left(\norm{x(0) -\hat{x}(0)} +\widetilde{\gamma}\right)\e{-\mu \ell T}.
\end{equation}
\end{proposition}
\begin{proof}
Using~\eqref{stop-crit} and~\eqref{capJ}, 
 we find that
\[\|\J^h_\ell\|\leq \sum_{n=1}^{N-1}\gamma \e{-\mu(N-n)\Delta T}\|J_{\ell,n}^{h}\| \leq  \widetilde{\gamma} \frac{\e{-\mu \ell T }}{2^\ell}.\]
Combining this inequality with~\eqref{ineq:parallel-error}
, we deduce that
\begin{align*}
\|\varepsilon_{\parallel}(T_{\ell})
 \|\leq &\left( \norm{ x(0)-\hat{x}(0)} +\sum_{j=1}^{\ell} \e{\mu j T}\cdot \widetilde{\gamma} \frac{\e{-\mu j T }}{2^j}\right) \gamma \e{-\mu\ell T} \\
=&  \gamma \left(\|\varepsilon_{\parallel}(0)\| + \widetilde{\gamma}\sum_{j=1}^{\ell}   \frac{1}{2^j}\right)\e{-\mu \ell T}\\
\leq &  \gamma \left(\|\varepsilon_{\parallel}(0)\| + \widetilde{\gamma}\right)\e{-\mu \ell T}.
\end{align*}
The result then follows from $\varepsilon_{\parallel}(0):=x(0) -\hat{x}_\parallel(0)=x(0) -\hat{x}(0)$.
\end{proof}
Note that~\eqref{stop-crit} actually deals with an \textit{a posteriori} quantity and, as such, can be used as a criterion to fix the level of accuracy of the time-parallelization procedure used in line on each window. We are now in a position to describe precisely our algorithm\FK{: the procedure for estimating 
$x(t)$ at $t=MT$ for some $M\in\N$ is detailed in Algorithm~\ref{diamond-strategy}.}
\begin{algorithm}
\caption{Diamond strategy}
\begin{algorithmic}
\STATE \textbf{Input:} $A,C,\hat{x_0},T,N,(\mu_i)_{i=1,\dots m},\widetilde{\gamma}, L$
\STATE \textbf{Output:}$(t_n^{\ell})_{\substack{n=0,\dots, N \\ \hspace{-0.7cm}\ell\in\N}},(\hat{X}_{n,\ell}^{h})_{\substack{n=0,\dots, N \\ \hspace{-0.7cm}\ell\in\N}}$ 
\STATE $L := \mathtt{place}(A,C,(\mu_i)_{i=1,\dots, m})$\; \COMMENT{place computes $L$, as in Theorem \ref{observability}}
\STATE $\mu := \min\limits_{\nu \in \sigma(A-LC)}{|\nu|}$\;
\STATE $\Delta t \vcentcolon = \frac{T}{N}$\;
\STATE $\ell := 0$\;
\REPEAT
	\STATE $T_{\ell} := \ell T$\; 
	\STATE $(\forall n \in \{1,\dots,N\}) \quad t_n^{\ell} := T_{\ell} + n\Delta T $ \;
	\IF{$\ell = 0$}
	\STATE{$\hat{X}_{\ell,0}^h := \hat{x}_0$}
	\ELSE 
	\STATE {$\hat{X}_{\ell,0}^h := \hat{x}_\parallel(T_{\ell-1}^-)$}  
	\ENDIF
	\STATE  Determine $h$ such that~\eqref{stop-crit} holds\; 
	\STATE  $\{\hat{X}_{\ell,n}^{h}\}_{n=1,\dots,N-1} := \mathtt{GTP}(\hat{X}_{\ell,0}^{h})$\;\COMMENT{Using a generic time-parallelization procedure (GTP)}
	\STATE Compute $\hat{x}_\parallel$ on $W_\ell$, by~\eqref{parallel-continuous}
	\STATE Assign $\ell \leftarrow \ell +1$	
\UNTIL{$\ell=M$}
\end{algorithmic}
\label{diamond-strategy}
\end{algorithm}


\section{Time Parallelization}\label{sec:parallel}
Note that Algorithm \ref{diamond-strategy} is defined independently of the chosen the parallel-in-time solver, since the jumps are computed \textit{a posteriori}. However, by specifying it, we can study in more detail the conditions that makes the criterion \eqref{stop-crit} satisfied and the complexity of the overall procedure. Indeed, having in hand an \textit{a priori} estimate of the jumps, one can determine the accuracy $h$ required on each window and bound the efficiency of the \textit{Diamond strategy}. In this way, we consider the Parareal algorithm as the time-parallel method (denoted by $\mathtt{GTP}$ in Algorithm~\ref{diamond-strategy}).

\subsection{The Parareal algorithm} 
Introduced by Lions, Maday and Turinici \cite{LMT}, the Parareal algorithm decomposes the \FK{solution} 
of an evolution problem by means of a partition of the considered bounded time interval. Assigning each of the corresponding subintervals to a processor, it alternately updates the initial conditions and solves the (smaller) problems on the subintervals in parallel, reducing the computational cost of the overall \FK{solution}. 
Let us describe \FK{the algorithm} 
more precisely. Given the problem 
\begin{equation}\left\{
\begin{aligned}
	\dot{u}(t) =&  f(u(t)),\quad t\in [0,T]\\
	u(0)	=&  u_0,
\end{aligned} \right. \label{parareal-example}
\end{equation}
decompose $[0,T]$ into a partition of $M$ subintervals $(t_{n-1},t_n)$. Consider then two solvers $\F$ and $\G$, that compute a fine and a coarse numerical approximation of $u$, respectively. The former is considered to be computationally expensive and consequently restricted to the (parallel) \FK{solution} 
of initial-value problems with high accuracy in each subinterval $(t_{n-1},t_n)$, whereas the latter is faster and can be used for solving (sequentially) on large intervals as $[0,T]$. For an arbitrary initial condition $\tilde{u}$ given in $t = t_{n-1}$, denote the corresponding local approximations of $u(t_{n})$ by $\F(t_{n},t_{n-1},\tilde{u})$ and $\G(t_{n},t_{n-1}, \tilde{u})$, respectively. In this framework, $\left(u(t_n)\right)_{n=1,\dots,M}$ is approximated by $(U_n^k)_{n=1,\dots,M}$, which is computed according to Algorithm \ref{parareal-algorithm}. 

\begin{algorithm}
\caption{Parareal algorithm}
\begin{algorithmic}
\STATE \textbf{Input: }{ $u_0,T,M,\mathrm{Tol}$}
\STATE \textbf{Output: }{ $(t_n)_{n=1,\dots,M},(U_{n}^{k^{*}})_{n=1,\dots,M}$}
\STATE $\Delta T \vcentcolon = \frac{T}{N}, t_0 := 0$\; 
\STATE $U_{0}^{0} \vcentcolon= u_0$\;\COMMENT{Initialization of the initial conditions}
\FOR{$1 \leq n \leq M$}
	\STATE $t_n \vcentcolon= n\Delta T $ \;
	\STATE $U_{n}^{0} \vcentcolon= \G(t_{n},t_{n-1},U_{n-1}^{0})$\;
\ENDFOR
\STATE $k \vcentcolon= 0$\;
\REPEAT
	\STATE$U_{0}^{k} \vcentcolon = u_0$\;	
	\FOR{$1 \leq n \leq M$}
	\STATE $U_{n}^{k+1} \vcentcolon= \F(t_{n},t_{n-1},U_{n-1}^{k})+\G(t_{n},t_{n-1},U_{n-1}^{k+1}) -\G(t_{n},t_{n-1},U_{n-1}^{k})$\;
	\STATE $J_{n}^{k} \vcentcolon= U_{n}^k -u(t_n^{-})$\;
	\STATE $k \leftarrow k +1$			
	\ENDFOR	
\UNTIL $\displaystyle{\max_{1\leq n\leq M} \norm{J_{n}^{k}} \leq \mathrm{Tol} }$
\STATE $k^{*} := k-1$	
\end{algorithmic}
\label{parareal-algorithm}
\end{algorithm}

Notice that the superscript $k$ in Algorithm \ref{parareal-algorithm} plays the role of the parameter $h$, introduced in the previous section. 

Gander and Vandewalle  show in \cite{GV} that the parareal algorithm reads as a multi-shooting algorithm, in the sense that the method is an approximate Newton method.
Indeed, solving the multiple shooting equations with the Newton's method yields
\[ U_{n}^{k+1} = u_{n-1}(t_{n},U_{n-1}^{k}) 
	+\dfrac{\partial u_{n-1}}{\partial U_{n-1}}(t_{n},U_{n-1}^{k})(U_{n-1}^{k+1}-U_{n-1}^{k}),\]
where $u_{n-1}(t_n,U_{n-1}^k)$ denotes the exact solution of	\eqref{parareal-example} at $t_{n}$, with initial condition $U_{n-1}^k$ at $t_{n-1}$. Approximating the exact solution $u_{n-1}(t_{n},U_{n-1}^{k})$ using the fine solver and the Jacobian term $\dfrac{\partial u_{n-1}}{\partial U_{n-1}}(t_{n},U_{n-1}^{k})(U_{n-1}^{k+1}-U_{n-1}^{k})$ 
by a difference on a coarse grid, gives 
		\begin{equation}	
			U_{n}^{k+1} = \F(t_{n},t_{n-1},U_{n-1}^{k})+\G(t_{n},t_{n-1},U_{n-1}^{k+1}) -\G(t_{n},t_{n-1},U_{n-1}^{k}).
		\label{parareal-update}
		\end{equation}
It follows that the convergence is super linear and that the number of iterations required to satisfy the criterion~\eqref{stop-crit} will not necessarily increase linearly with respect to $M$. 
 In addition, these authors obtain 
an estimate which shows that after $k$ iterations, the algorithm gives rise to the fine solution on the first $k$ subintervals. An improvement of their estimate, due to Gander and Hairer~\cite{GH}, assumes that the coarse solver must satisfy a Lipschitz condition 
\begin{equation*}
	\norm{\G(t_{n},t_{n-1},y) -\G(t_{n},t_{n-1},z)} 
	\leq (1+c \Delta T)\norm{y-z},
\end{equation*}
for a positive constant $c$. However, this result 
\FK{does not capture the enhanced convergence rate when the differential equation
itself exhibits decaying behaviour, i.e., when $c<0$.} Since we are interested in coupling this algorithm with the Luenberger observer and taking advantage of its decaying behavior, we provide a result adapted to this case, which follows from \cite{GKZ2018}. 

\begin{theorem}[Convergence of the Parareal algorithm for decaying problems]\label{parareal-convergence} 
Given an initial condition $z$ at time $t_{n-1}$, 
let $\F(t_{n},t_{n-1},z)$ and $\G(t_{n},t_{n-1},z)$ be be the exact solution at time $t_{n}$ and its approximation by a coarse integrator respectively. Assume that the local truncation error 
$\tau(t_{n},z) \vcentcolon = \F(t_{n},t_{n-1},z) - \G(t_{n},t_{n-1},z)$ satisfies for all $y$ and $z$
\begin{align}
	\norm{\tau(t_{n},z)}	\leq & \alpha, \label{condalpha}\\
	\norm{\tau(t_{n},y) -\tau(t_{n},z)} 	\leq & \beta\norm{y-z},\label{condbeta}
\end{align}
where $\alpha,\beta>0$ are constants, and that $\F$ and $\G$ are Lipschitz with respect to the initial conditions:
\begin{equation}
\max\left\{
		\norm{\F(t_{n},t_{n-1},y) -\F(t_{n},t_{n-1},z)},
		\norm{\G(t_{n},t_{n-1},y) -\G(t_{n},t_{n-1},z)}\right\} 
	\leq \eta\norm{y-z},\label{condepsi}
\end{equation}
for a constant $\eta \in (0,1)$. 
The error $\norm{U_n^{k} - u(t_n)}$ at iteration $k$ is bounded by $B_{n}^{k}$, defined by
\begin{equation}\label{kwok}
	B_{n}^{k} := 
		\begin{cases}
			0 & n \leq k \\
			\alpha \beta^k \sum\limits_{i=0}^{n-k-1}\tbinom{k+i}{k} \varepsilon^i & n > k.
		\end{cases}
\end{equation}
\end{theorem}
\JS{
\begin{remark}
Suppose that~\eqref{condalpha} and~\eqref{condepsi} are satisfied for some $\alpha>0$ and $\eta\in(0,1)$, then~\eqref{condbeta} holds with $\beta:=2\eta$, so that~\eqref{condbeta} seems unnecessary. However, using $\beta:=2\eta$ does not necessarily give useful bounds when $0.5 < \eta < 1$.
\end{remark}
}
\begin{proof}
Using Definition \eqref{parareal-update} and the fact that $\F$ corresponds to the exact solution on $(t_{n-1},t_n)$, we obtain
\begin{align*}
U_n^{k} - u(t_n) 
	=& \F(t_{n},t_{n-1},U_{n-1}^{k-1}) +\G(t_{n},t_{n-1},U_{n-1}^{k}) -\G(t_{n},t_{n-1},U_{n-1}^{k-1}) \\
	&  -\F(t_{n},t_{n-1},u(t_{n-1})) \\
	=& \tau(t_{n}, U_{n-1}^{k-1}) - \tau(t_{n},u(t_{n-1})) +\G(t_{n},t_{n-1},U_{n-1}^{k}) - \G(t_{n},t_{n-1},u(t_{n-1})).
\end{align*}
Taking norms and combining the resulting inequality with~\eqref{condbeta} and~\eqref{condepsi} gives
\begin{equation*}
\norm{U_n^{k} - u(t_n)} 
	\leq \beta\norm{U_{n-1}^{k-1} - u(t_{n-1})} + \varepsilon
	\norm{U_{n-1}^{k} - u(t_{n-1})}.
\end{equation*}
The error in the initial condition can be estimated similarly. We have
\begin{align*}
U_n^{0} - u(t_n) =& \G(t_{n},t_{n-1},U_{n-1}^{0}) - \F(t_{n},t_{n-1}u(t_{n-1})) \\
& + \G(t_{n},t_{n-1},u(t_{n-1})) - \G(t_{n},t_{n-1},u(t_{n-1})) + \F(t_{n},t_{n-1},u(t_{n-1})) \\
\leq& \alpha +\eta\norm{U_{n-1}^0-u(t_{n-1})},
\end{align*}
which gives, by means of~\eqref{condalpha} and~\eqref{condepsi}
\[\norm{U_n^{0} - u(t_n)}\leq \alpha + \eta \norm{U_{n-1}^{0} - u(t_{n-1})}. \]
Therefore, an upper bound $B_n^{k}$ for $\norm{U_n^{k}-u(t_n)}$ satisfies the recurrence relation
\begin{align}
	B_n^{k} 	=&  \beta B_{n-1}^{k-1}+ \eta B_{n-1}^{k}, \label{recur-1} \\
	B_n^0 		=&  \alpha + \eta B_{n-1}^0, \label{recur-2}
\end{align} 
with $B_0^k = 0$ for all $k$. This recurrence can be solved by means of generating functions, namely, by defining the formal power series
\[ \rho_k(\zeta) = \sum_{n\geq 1}B_n^k \zeta^n. \]
Multiplying \eqref{recur-1} and \eqref{recur-2} by $\zeta^n$ and summing over $n\geq 1$ gives
\begin{align*}
\rho_{k}(\zeta)	=&  \beta\zeta \rho_{k-1}(\zeta) + \eta\zeta \rho_{k}(\zeta), \\
\rho_0(\zeta)	
	=&  \frac{\alpha\zeta}{1-\zeta} + \eta\zeta \rho_0(\zeta),
\end{align*}
which can be solved by induction and yields the explicit formula
\[ 
\rho_k(\zeta) = \frac{\alpha \beta^k\zeta^{k+1}}{(1-\zeta)(1-\eta\zeta)^{k+1}}. 
\]
Expanding $\rho_k(\zeta)$ in a power series leads to 
\begin{align*}
\rho_k(\zeta) 
	=&  \alpha \beta^k \zeta^{k+1}
		\Big(\sum_{i\geq0} \zeta^i\Big)
		\Big(\sum_{j\geq 0} \tbinom{k+j}{k} (\eta\zeta)^j\Big)
 	= \alpha \beta^k \zeta^{k+1}
 		\sum_{n\geq 0}\Big(\sum_{i=0}^{n}\tbinom{k+i}{k} \eta^i\Big)\zeta^{n}\\
 =& \sum_{n\geq 0}\Big(\alpha \beta^k\sum_{i=0}^{n}\tbinom{k+i}{k}\eta^i\Big)\zeta^{n+k+1}.
\end{align*}
Then, for $n\leq k$ we have $B_0^k = \ldots = B_k^k = 0$; whereas for $n > k$, we obtain
\[ 
B_n^{k} = \alpha\beta^k\sum_{i=0}^{n-k-1}\tbinom{k+i}{k} \eta^i,
\]
and the result follows.
\end{proof}
We can derive from the previous result an estimate on the jumps.
\begin{corollary}\label{parareal-bounded-jumps}
After $k$ iterations of Algorithm~\ref{parareal-algorithm}, the jump $\widetilde{J}_n^k := U_n^k -u(t_n^{-})$ satisfies $\norm{\widetilde{J}_n^k} \leq 2 B_n^k$. 
\end{corollary}

\begin{proof}
Since $\F$ is an exact solver, we have
\begin{align*}
\norm{\widetilde{J}_n^k} 
    =&  \norm{U_n^k -\F(t_{n},t_{n-1},U_{n-1}^{k})} \\
    \leq & \norm{U_n^k-u(t_n)} 
    +\norm{u(t_n) -\F(t_{n},t_{n-1},U_{n-1}^{k})} \\
    \leq & B_n^k 
    +\norm{\F(t_{n},t_{n-1},u(t_{n-1})) -\F(t_{n},t_{n-1},U_{n-1}^{k})} \\
    \leq & B_n^k + \eta\norm{u(t_{n-1}) -U_{n-1}^k}
    = B_n^k + \eta B_{n-1}^k \leq 2 B_n^k,
\end{align*}
\JS{since $\eta\in(0,1)$ and $ B_{n-1}^k \leq B_n^k$.}
The result follows from~\eqref{recur-1}. 
\end{proof}

\subsection{Combination with Luenberger observer} 
We now use the Parareal scheme to define the initial conditions $(\hat{X}_{n,\ell}^{k})_{\substack{n=0,\dots, N \\ \hspace{-0.7cm}\ell\in\N}}$ of the time-parallelized Luenberger observer~\eqref{parallel-continuous}, meaning that 
\begin{equation}\left\{
\begin{aligned}
	\hat{X}_{\ell,n}^{k} =&  \F(t_{n}^{\ell},t_{n-1}^{\ell},\hat{X}_{\ell,n-1}^{k-1})+\G(t_{n}^{\ell},t_{n-1}^{\ell},\hat{X}_{\ell,n-1}^{k}) -\G(t_{n}^{\ell},t_{n-1}^{\ell},\hat{X}_{\ell,n-1}^{k-1}) \\
	\hat{X}_{\ell,n}^{0} =&  \G(t_{n}^{\ell},t_{n-1}^{\ell},\hat{X}_{\ell,n-1}^{0}),\quad \hat{X}_{\ell,0}^{0} = \hat{x}_{\parallel}(T_{\ell}^{-}).
\end{aligned}\right. \label{parareal-luenberger}
\end{equation}
Recall that in this setting, jumps are obtained during the execution of Algorithm~\ref{parareal-algorithm} and are consequently known \textit{a posteriori}.  
 In order to estimate the efficiency of Algorithm~\ref{diamond-strategy}, we now propose on the contrary to derive \textit{a priori} upper bounds of the number of iterations  observed in practice and complexity. 
Given a tolerance parameter $\mathrm{Tol}$ and let define the corresponding efficiency of Algorithm~\ref{diamond-strategy} by 
\begin{equation}\label{efficiency}
	E := \dfrac{\tau_s}{N \tau_p} 
\end{equation}
where $\tau_s$ and $\tau_p$ are the CPU time required to achieve $\|\varepsilon(t)\|\leq \mathrm{Tol}$ using a sequential solver and $\|\varepsilon_{\parallel}(t)\|\leq \mathrm{Tol}$ using a parallel solver, respectively. Recall that $N$ represents the number of available processors (and hence, subintervals).
\begin{theorem}\label{efficiency-theorem}
We neglect the time of interprocessor \FK{communication} 
in the computational time.  
Let $\tau_{\cal G}$ and $\tau_{\cal F}$ be the computational times associated with one coarse and one fine \FK{solution} 
of~\eqref{Luenberger} on a interval of length $T$.  
The efficiency of the algorithm satisfies
\begin{equation}\label{parareal-efficiency}
	E \FK{\geq}  \dfrac{\ell^{\rm Tol}\tau_{\cal F}}{\tau_{\cal F} + N\tau_{\cal G}  } \left(\sum_{\ell = 1}^{\ell_{\parallel}^{\rm Tol}} k_{\ell}\right)^{-1},
\end{equation}
where 
\begin{align*}
k_\ell :=& \min_k\left\{k \FK{\;:\;} 2\gamma\sum_{n=k+1}^{N-1} \e{-\mu(N-n)\Delta T}B_n^{k} \leq \widetilde{\gamma} \dfrac{\e{-\mu \ell T}}{2^{\ell}}  \right\},\\
\ell^{\rm Tol}:=&\left\lceil\frac 1{\mu T} \log \left(\gamma\frac{\norm{x(0) -\hat{x}(0)}}{\rm Tol}\right)\right\rceil, \\
\ell_{\parallel}^{\rm Tol}:=& \left\lceil\frac 1{\mu T} \log \left(\gamma\frac{\norm{x(0) -\hat{x}(0)}+\widetilde{\gamma}}{\rm Tol}\right)\right\rceil,
\end{align*}
where $\lceil . \rceil$ denotes the ceiling function. 
\end{theorem}
\JS{
The integers $\ell^{\rm Tol}$ and $\ell_{\parallel}^{\rm Tol}$ are upper bounds for the number of windows required to obtain
$\norm{\varepsilon_{\parallel}(T_\ell)}\leq {\rm Tol}$ and
$\norm{\varepsilon_{\parallel}(T_\ell)}\leq {\rm Tol}$, respectively. 
Let us denote by  $k_\ell^{obs}$ the actual number of iterations performed in Algorithm~\ref{diamond-strategy} in the window $W_\ell$. We will see that $k_\ell^{obs}$ is bounded by $k_\ell$.
   }

\begin{proof}
Using Corollary~\ref{parareal-bounded-jumps}, we find that at if $k_\ell$ iterations of~\eqref{parareal-luenberger} are done by Algorithm~\ref{diamond-strategy} in the window $W_\ell$, the left-hand side of \eqref{stop-crit} satisfies
\begin{align*}
\gamma\sum_{n=1}^{N-1} \e{-\mu(N-n)\Delta T}\|J_{\ell,n}^{k_\ell}\| 	\leq & 2\gamma\sum_{n=1}^{N-1} \e{-\mu(N-n)\Delta T}B_n^{k_\ell} \\
=&  2\gamma\sum_{n=k_\ell+1}^{N-1} \e{-\mu(N-n)\Delta T}B_n^{k_\ell} 
\leq \widetilde{\gamma} \dfrac{\e{-\mu \ell T}}{2^{\ell}},
\end{align*}
meaning that~\eqref{stop-crit} is satisfied. As a consequence, we have
\begin{equation}\label{estk}
k_\ell^{obs}\leq k_\ell.
\end{equation}
Because of~\eqref{luenberger-decay} and~\eqref{Luencv}, the number of windows required to get $\|\varepsilon(t)\|\leq \mathrm{Tol}$ and $\|\varepsilon_{\parallel}(t)\|\leq \mathrm{Tol}$ are bounded by $\ell^{\rm Tol}$ and $\ell_{\parallel}^{\rm Tol}$, respectively. 
 In view of~\eqref{parareal-luenberger} and  since we neglect the time of interprocessor \FK{communication,} 
 the part of the total computational time spend required by Algorithm~\ref{diamond-strategy} to deal with the window $W_\ell$ 
 \FK{is given by}
$\tau_{p,\ell}= k_\ell^{obs} (\frac{\tau_{\cal F}}{N} + \tau_{\cal G}  ).$
 On the other hand, the fine solver needs $\tau_{s,\ell}=\tau_{\cal F}$ to 
\FK{complete one solve of~\eqref{Luenberger} on $W_\ell$ on the fine grid.}
\JS{Because of~\eqref{estk}}, the efficiency 
satisfies
\[	E  \;\FK{ \geq \dfrac{\ell^{\rm Tol}\tau_{\cal F}}{\tau_{\cal F} + N\tau_{\cal G}  } \left(\sum_{\ell = 1}^{\ell_{\parallel}^{\rm Tol}} k_{\ell}^{obs}\right)^{-1} \geq } \; \dfrac{\ell^{\rm Tol}\tau_{\cal F}}{\tau_{\cal F} + N\tau_{\cal G}  } \left(\sum_{\ell = 1}^{\ell_{\parallel}^{\rm Tol}} k_{\ell}\right)^{-1},
\]
which is the desired estimate.
\end{proof}

\subsection{Variable window approach} 
Using the results of the previous section, we can propose a variant of Algorithm~\ref{diamond-strategy} devoted to the case of a large number of processors. 
Instead of 
\FK{always using the same window length $T$, we now choose the window length $T_\ell$ as a function of a
prescribed number of iterations $k_\ell$, in a way that ensures 
that the error estimate in Corollary~\ref{parareal-bounded-jumps} 
falls below the given tolerance after $k_\ell$ iterations.
Since the parareal error must decrease at the same rate as the 
assimilation error as $t$ increases, the number of iterations $k_\ell$ must
increase with $\ell$; therefore,}
we \FK{will} fix the number of parareal iterations applied on each window  to $k'_{\ell} := \ell$, \FK{and determine the window length $T_\ell$
\textit{a priori} using Corollary~\ref{parareal-bounded-jumps}.} For the sake of clarity, we denote by $(W'_\ell)_{\ell\in\N}$ the corresponding set of windows. Suppose that the window $W'_\ell$ is composed of $N_\ell$ subintervals of lengths $\Delta T$, i.e., $W'_{\ell} = \bigcup_{n=0}^{N_\ell-1} (t_n^{\ell},t_{n+1}^{\ell})$ and define, for $\ell\in\N$, $T'_\ell:=t_0^{\ell+1}=t_{N_{\ell}}^\ell$, meaning that $T'_\ell=\sum_{j=1}^\ell N_j \Delta T$ if $\ell>0$ and $T'_0=0$. Since the number of iterations is now fixed for each window, we simply denote by $J_{\ell,n}$ (instead of $J^\ell_{\ell,n}$) the jumps observed at $t_n^\ell$. Lemma~\ref{parallel-error} then translates as follows. 

\begin{lemma}\label{parallel-error-var}
Let $\ell>0$. Under the assumptions of Proposition~\ref{stability} and still denoting by $\varepsilon_{\parallel}(t):= x(t)-\hat{x}_{\parallel}(t)$ the error associated with the approximation~\eqref{parallel-continuous}, we have
\begin{equation}
\norm{\varepsilon_{\parallel}(T'_\ell)}
	\leq \gamma \left( \norm{ x(0)-\hat{x}(0)} +\sum_{j=1}^{\ell} \e{\mu  T_j'}\| {\J_j}'\|\right)\e{-\mu T_\ell'} \label{ineq:parallel-error-var}
\end{equation}
where 
\begin{equation}\label{capJvar}
{\J_\ell}':=-\sum_{n=1}^{N_\ell-1} \e{(A-LC)(N_\ell-n)\Delta T}J_{\ell,n}.
\end{equation}
\end{lemma}
\begin{proof}
As in the proof of Lemma~\ref{parallel-error}, we have
\begin{align*}
\varepsilon_{\parallel}(\tnlm)
    = \e{(A-LC)\Delta T}(\varepsilon_{\parallel}(t^{\ell -}_{n-1})-J_{\ell,n-1}  ),
\end{align*}
hence
\begin{equation*}
\varepsilon_{\parallel}(T_{\ell}') =\varepsilon_{\parallel}(t^{\ell}_{N_\ell})
= \e{(A-LC)(T_{\ell}'-T_{\ell-1}')}\varepsilon_{\parallel}(t^{\ell}_{0})+\J_\ell'= \e{(A-LC)(T_{\ell}'-T_{\ell-1}')}\varepsilon_{\parallel}(T_{\ell-1}')+\J_\ell',
\end{equation*}
where we have used the continuity of $\hat{x}_{\parallel}$ in $T_\ell'=t_{N_\ell}^\ell$ and $T_{\ell-1}'=t_{N_\ell}^{\ell-1}$.
In the same way, we obtain
\[\varepsilon_{\parallel}(T_{\ell})=\e{(A-LC) T_\ell' }\varepsilon_{\parallel}(0) + \sum_{j=1}^{\ell} \e{(A-LC)(T_\ell'-T_j')} \J_j'.\]
The result is obtained by taking the norm and using Proposition \ref{stability}.
\end{proof}
The rate of convergence can now be preserved \textit{a priori}.
\begin{proposition}\label{bound-jumps-var}
Given $\widetilde{\gamma}>0$ an arbitrary parameter, define $N_\ell$ recursively by
\begin{equation}\label{defNl}
N_{\ell}
:= \max \left\{N \; : \; 2\gamma\sum_{n=\ell
}^{N-1} \e{\mu n\Delta T}B_n^{\ell} \leq \widetilde{\gamma} \dfrac{\e{-\mu T_{\ell-1}'}}{2^{\ell}}  \right\}.
\end{equation}
Then, the rate of convergence of $\hat{x}_{\parallel}(t)$ to $x(t)$ is bounded by $\mu$, i.e.
\begin{equation}\label{Luencv-var}
\norm{\varepsilon_{\parallel}( T_{\ell}')} \leq \gamma \left(\norm{x(0) -\hat{x}(0)} +\widetilde{\gamma}\right)\e{-\mu  T_\ell'}.
\end{equation}
\end{proposition}
Note that $B_n^\ell=0$ for $n\leq \ell$, so that $\sum_{n=\ell
}^{N-1} \e{\mu n\Delta T}B_n^{\ell}=0$ if $N_\ell = \ell + 1$. Hence, $N_\ell \geq \ell +1$. 
\begin{proof}
Using successively~Proposition~\ref{stability}, ~\eqref{capJvar}, Corrollary~\ref{parareal-bounded-jumps} and \eqref{defNl}, we find that:
\begin{align*}
\|{\J_\ell}'\|\leq &\gamma \sum_{n=1}^{N_\ell-1} \e{-\mu(N_\ell-n)\Delta T}\|J_{\ell,n}\| \\
\leq & 2\gamma\sum_{n=1}^{N_\ell-1} \e{-\mu(N_\ell-n)\Delta T}B^\ell_n =  2\gamma\sum_{n=1}^{N_\ell-1} \e{\mu n\Delta T}B^\ell_n \e{-\mu (T_\ell'-T_{\ell-1}')} \\
\leq &  \widetilde{\gamma} \dfrac{\e{-\mu T_{\ell}'}}{2^{\ell}}.
\end{align*}
Combining this last equation with~\eqref{ineq:parallel-error-var}, we get:
\begin{align*}
\norm{\varepsilon_{\parallel}(T'_\ell)}
	\leq & \gamma \left( \norm{ x(0)-\hat{x}(0)} +\sum_{j=1}^{\ell} \e{\mu  T_j'}\widetilde{\gamma} \dfrac{\e{-\mu T_{j}'}}{2^{j}}\right)\e{-\mu T_\ell'} \\
	\leq & \gamma \left( \norm{ x(0)-\hat{x}(0)} +\widetilde{\gamma}\sum_{j=1}^{\ell}  \dfrac{1}{2^{j}}\right)\e{-\mu T_\ell'} \\
	\leq & \gamma \left( \norm{ x(0)-\hat{x}(0)} +\widetilde{\gamma}\right)\e{-\mu T_\ell'}
\end{align*}
The result then follows from $\varepsilon_{\parallel}(0):=x(0) -\hat{x}_\parallel(0)=x(0) -\hat{x}(0)$.
\end{proof}

\section{Numerical experiments}\label{sec:numerics}
The present section is devoted to some numerical experiments for the Luenberger observer.   
For this purpose, we use
\begin{equation*}
A = \begin{bmatrix}0 & 1 \\ -1 & -2\end{bmatrix},\; 
B = \begin{pmatrix} 0 \\ 1 \end{pmatrix},\; 
C = \begin{pmatrix} 0 \\ 1 \end{pmatrix},\; 
v(t) = 3 +0.5\sin(0.75t).
\end{equation*}
We remark that the initial condition on System \eqref{det-assim} is unknown, but we perform the experiments with $x(0) = (0\FK{,} \; 0)^{\t}$. We then construct the observer $\hat{x}(t)$ by setting as initial condition $\hat{x}(0) = (2\FK{,} \; 1)^{\t}$ and the eigenvalues of $A-LC$. For the latter, we consider $\{-0.25,-0.5\}$ and $\{-2,-4\}$ as possible choices.

To introduce the parareal procedure, given $N$ available processors, we set
\[
T = 1,\; \delta T = \Delta T = \dfrac{T}{N},\; \textrm{Tol} = 10^{-8},
\]
where $\delta T$ denotes the time step associated with $\G$, chosen as a one step solver for the sake of simplicity. We use the Backward Euler method to define both propagators $\F$ and $\G$.

\subsection{Diagonalized system} 
We recall that the essential part of Theorem \ref{parareal-convergence} is the  \FK{contraction} 
factor $\eta$. For the Luenberger observer \eqref{Luenberger}, we have 
\begin{equation*}
\eta = \max\left\{ 
	\norm{[I -\delta t (A-LC)]^{-\sfrac{\Delta T}{\delta t}}},
	\norm{[I - \Delta T(A-LC)]^{-1}}	\right\}.
\end{equation*}
where $\delta t$ is the time step associated with $\F$, assumed to be constant. Even if we choose the eigenvalues of $A-LC$ to guarantee a decaying rate of convergence, $\eta$ is not necessarily smaller than one. For this reason, we consider instead a diagonalized observer 
\begin{equation}\label{Luenberger-diag}
\left\{
\begin{aligned}
	\dot{\hat{z}}(t) =&  D\hat{z}(t) +V^{-1}(Bu(t) +Ly(t)) \\
	\hat{z}(0) =&  V^{-1}\hat{x}_0 \\
\end{aligned}\right.
\end{equation}
where $\hat{z} = V^{-1}\hat{x}$ and $D = V^{-1}(A-LC)V$. 

Due to the change of variables, $\gamma = 1$. We determine the constants $\alpha$, $\beta$ and $\eta$ by
\begin{proposition}\label{BWDEuler} 
Let $\F$ and $\G$ be defined by the Backward Euler scheme, with time steps $\delta t$ and $\delta T$, respectively. We assume that $\Delta T K \leq 1$ and~\eqref{Luenberger-diag} satisfies
\begin{align*}
	M & \vcentcolon = \sup_{(\hat{z},t)}\norm{D\hat{z} +V^{-1}(Bu(t)+Ly(t))} < \infty \\
	K & \vcentcolon = \max\left\{\norm{D},\sup_{t>0}\norm{V^{-1}(B\dot{u}(t) +L\dot{y}(t))}\right\} < \infty .
\end{align*}
Then, the constants associated with both propagators in Theorem \ref{parareal-convergence} are given by
\begin{align}
	\alpha  	=&  \Delta T^2\left(\dfrac{K(M+1)}{2(1-\Delta T K)}\right), \nonumber \\
	\beta	=&  \norm{[I -\delta tD]^{-\sfrac{\Delta T}{\delta t}} -[I -\Delta T D]^{-1}}, \label{BE-beta} \\
\eta =&  \max\left\{ 
	\norm{[I -\delta t D]^{-\sfrac{\Delta T}{\delta t}}},
	\norm{[I - \Delta T D]^{-1}}	\right\} \label{BE-epsilon}.
\end{align}
\end{proposition}
The proof is standard, but for the sake of completeness is presented in Appendix~\ref{appA}. 

\subsection{Evolution of $k_{\ell}$}
As a first experiment, since the jumps involved in~\eqref{stop-crit} allows us to compute the sequence $k^{obs} := \{k_{\ell}^{obs}\}_{\ell}$, we propose to compare its behavior with its \emph{a priori} estimate 
\[k^{th} := \{k_{\ell}\}_{\ell},\]
where the latter sequence is provided by Theorem~\ref{efficiency-theorem}. \\

We observe in Figure \ref{Exp1} that increasing $\widetilde{\gamma}$ leads to enlarge the number of windows in which the algorithm requires only 1 iteration. This is expected, due to the term $\widetilde{\gamma}\e{-\mu \ell T}$ present in Proposition~\ref{bound-jumps}. 

\begin{figure}
	\begin{subfigure}{\textwidth}
	\includegraphics[height=5.5cm,width=0.475\textwidth]{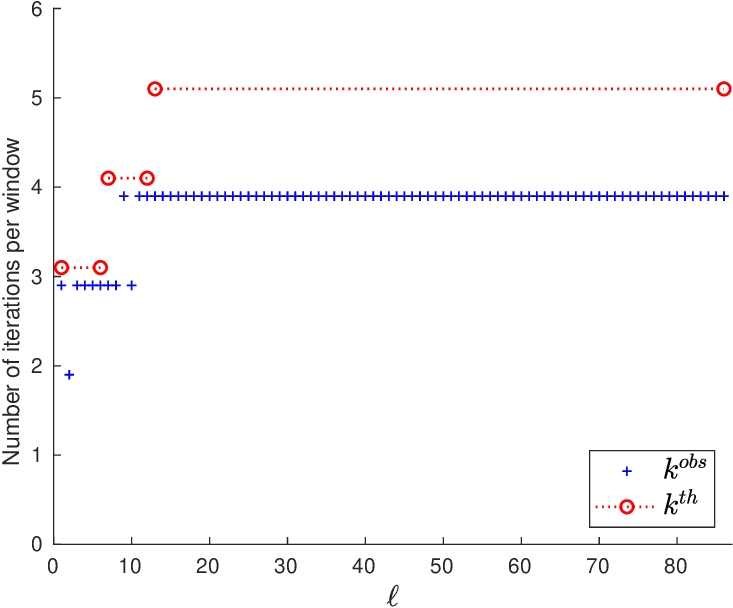}
	\hfill
	\includegraphics[height=5.5cm,width=0.475\textwidth]{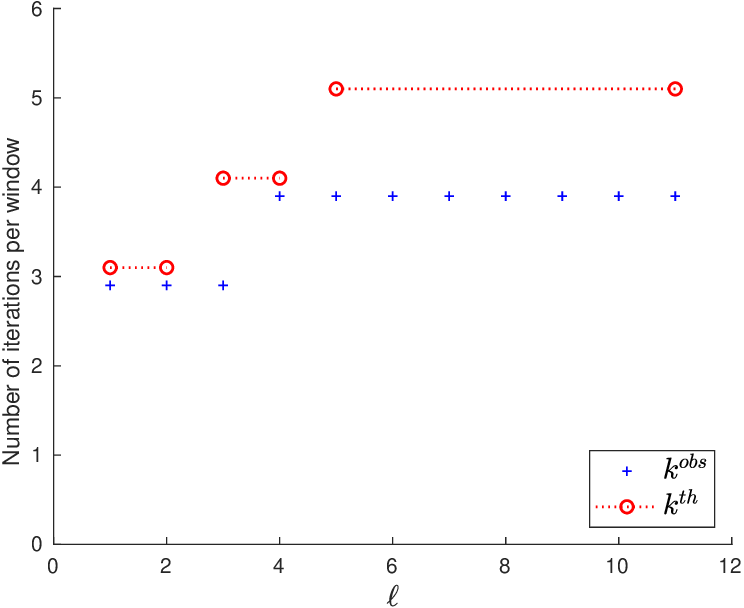}
  	\caption{$\widetilde{\gamma} = 10^{-3}$}
  	\label{exp1a}
	\end{subfigure}	
	\vskip\baselineskip
	\begin{subfigure}{\textwidth}
	\includegraphics[height=5.5cm,width=0.475\textwidth]{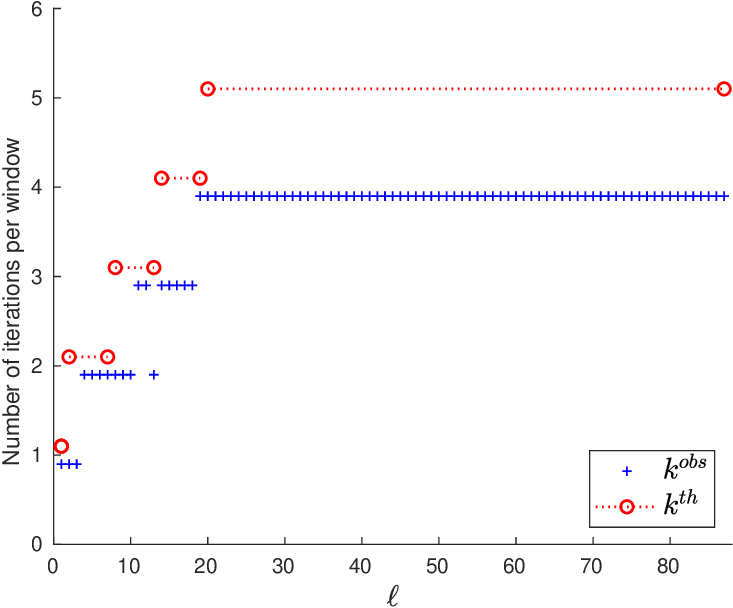}
	\hfill
	\includegraphics[height=5.5cm,width=0.475\textwidth]{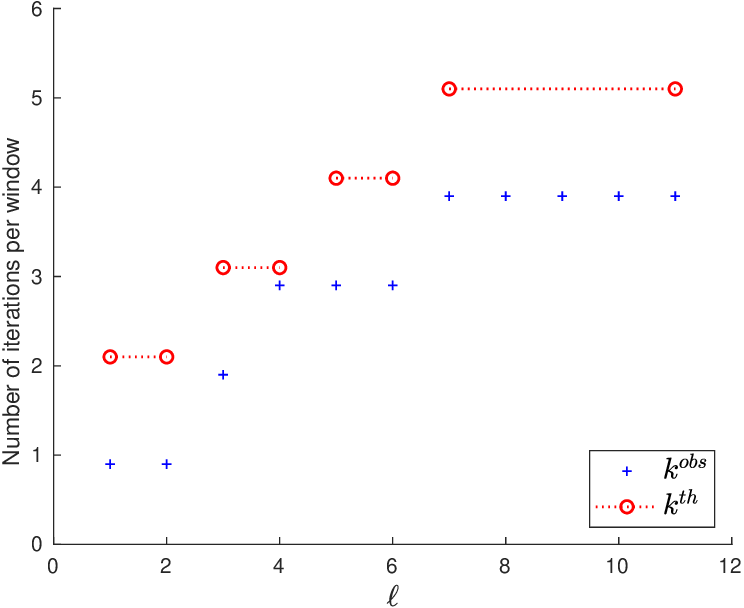}
  	\caption{$\widetilde{\gamma} = 1$.}
  	\label{exp1b}
	\end{subfigure}
	\vskip\baselineskip
	\begin{subfigure}{\textwidth}
	\includegraphics[height=5.5cm,width=0.475\textwidth]{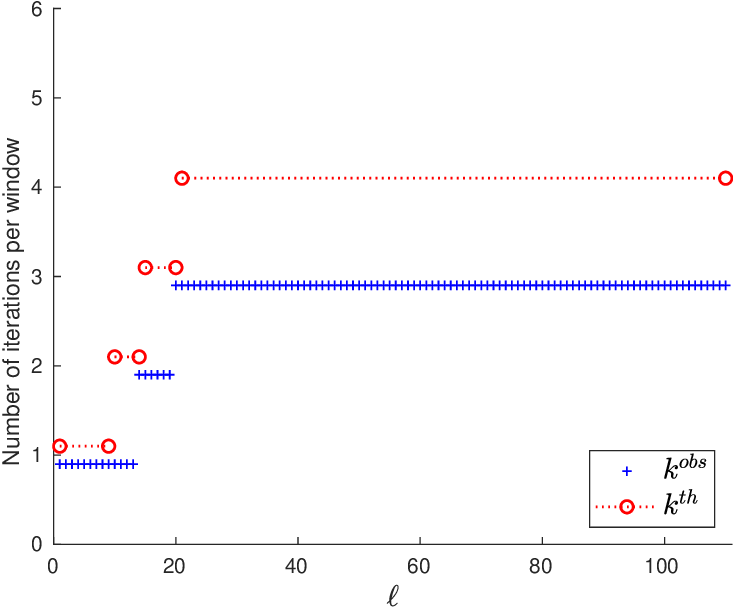}
	\hfill
	\includegraphics[height=5.5cm,width=0.475\textwidth]{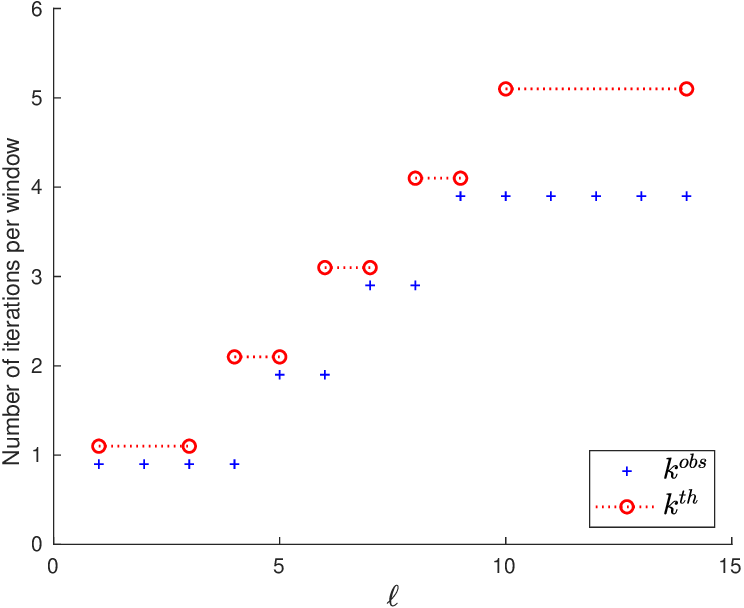}
  	\caption{$\widetilde{\gamma} = 10^{3}$.}
  	\label{exp1c}
	\end{subfigure}
\caption[Comparison between $k^{th}$ and $k^{obs}$]{Comparison between $k^{th}$ and $k^{obs}$, for $N = 16$ and $\delta t = \frac{\Delta T}{2^5}$. The eigenvalues of $A-LC$ are $\{-0.25,-0.5\}$ (left) and $\{-2,-4\}$ (right).}
\label{Exp1}
\end{figure}


\subsection{Observed efficiency}
Our second experiment consists \FK{of} comparing the observed efficiencies for both sequences $k^{obs}$ and $k^{th}$, using different values of $\widetilde{\gamma}$, $N$ and $\delta t$. To evaluate $E^{obs}$, the execution time for the parallel and sequential solvers was computed with the functions \texttt{tic} and \texttt{toc} of \texttt{MATLAB} (version 9.4.0.813654 (R2018a)).

As we notice previously, increasing $\widetilde{\gamma}$ improves the algorithm performance, but the difference between $E^{obs}(k^{obs})$ and $E^{obs}(k^{th})$ still remains, as observed in Figure \ref{exp2-1}. 
In Figure \ref{exp2-3}, the gap between these values varies slightly, showing that $\delta t$ small enough does not affect the efficiency significantly. 
Increasing the number of processors $N$ makes this difference smaller and also improves the efficiency of the algorithm, as shown in Figure \ref{exp2-2}. Another way to narrow this gap is choosing smaller eigenvalues for $A-LC$. As Figure \ref{Exp2} suggests, the comparison between $\{-0.25,-0.5\}$ and $\{-2,-4\}$ shows that $E^{obs}(k^{th})$ increases, whereas $E^{obs}(k^{obs})$ becomes smaller.

Figure \ref{Exp2} also shows that the observed efficiencies satisfy 
\[E^{obs}(k^{th}) \leq E^{obs}(k^{obs}), \]
which is simply because the sequence $k^{th}$ underperforms $k^{obs}$. 


Finally, we recall that $k^{th}$ is useful for estimating the efficiency. Assuming that $\tau_{\Delta T}^{\G}$ is negligible, we denote this estimate by 
\begin{equation*}
	E_{0}^{th} 
	= \ell_{\parallel}^{\rm Tol}\left(\sum_{\ell = 1}^{\ell_{\parallel}^{\rm Tol}} k_{\ell}\right)^{-1}.
\end{equation*}
with $\ell_{\parallel}^{\rm Tol}$ given by Theorem \ref{efficiency-theorem}.  We note that this value predicts quite well $E^{obs}(k^{th})$ in all the tests. 

\begin{figure}
	\begin{subfigure}{\textwidth}
	\includegraphics[height=5.5cm,width=0.475\linewidth]{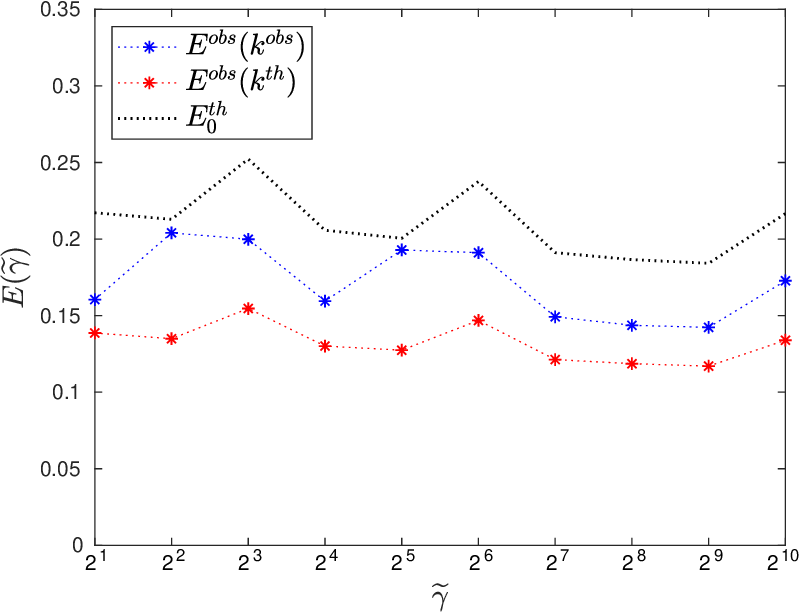}
	\hfill
	\includegraphics[height=5.5cm,width=0.475\linewidth]{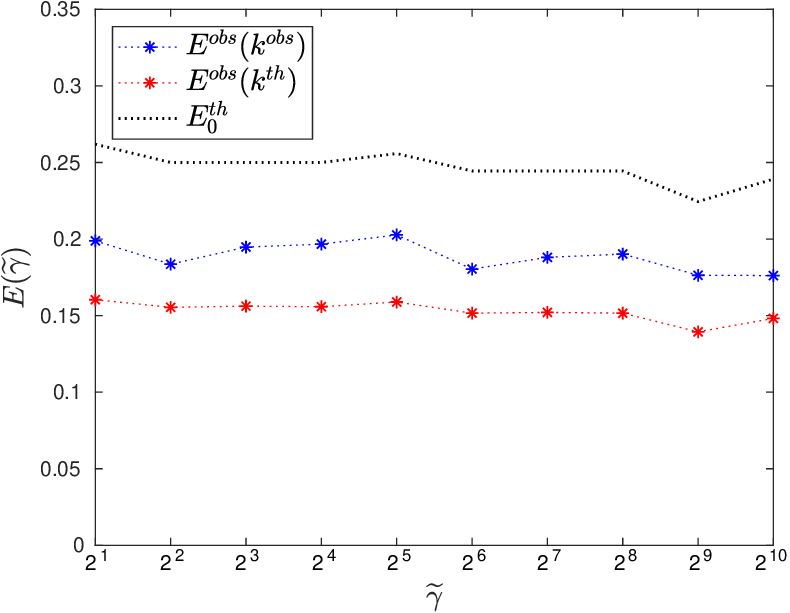}
  	\caption{$E(\widetilde{\gamma})$, for $N = 16$ and $\delta t = \frac{\Delta T}{2^5}$.}
  	\label{exp2-1}
	\end{subfigure}	
	\vskip\baselineskip
	\begin{subfigure}{\textwidth}
	\includegraphics[height=5.5cm,width=0.475\linewidth]{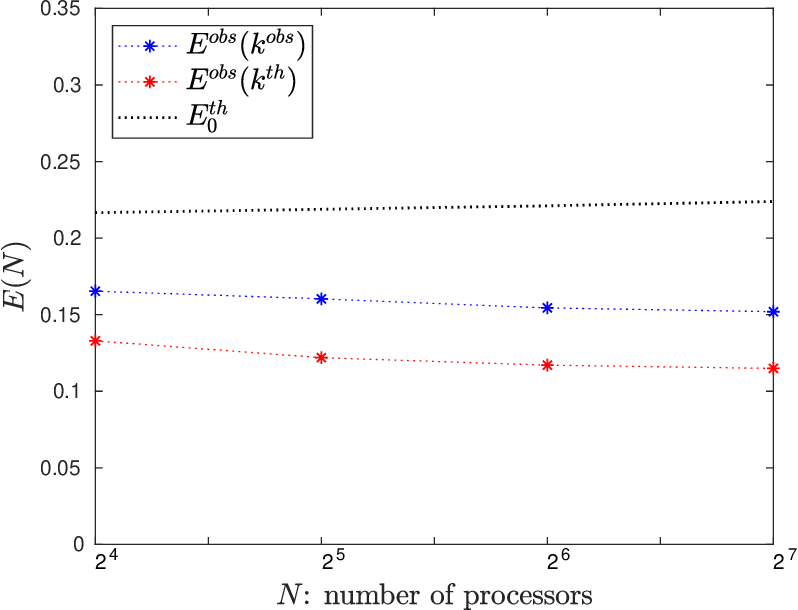}
	\hfill
	\includegraphics[height=5.5cm,width=0.475\linewidth]{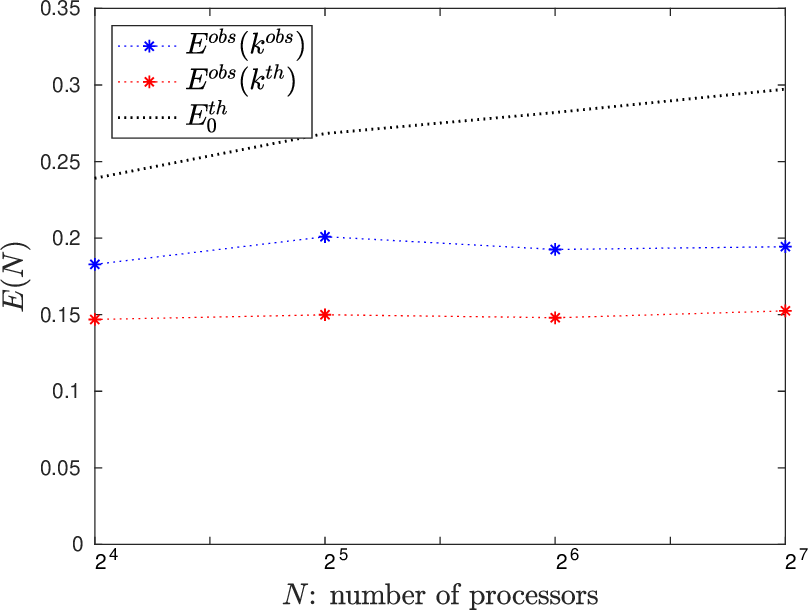}
  	\caption{$E(N)$, for $\delta t = \frac{\Delta T}{2^5}$ and $\widetilde{\gamma} = 2^{10}$.}
  	\label{exp2-2}
	\end{subfigure}
	\vskip\baselineskip
	\begin{subfigure}{\textwidth}
	\includegraphics[height=5.5cm,width=0.475\linewidth]{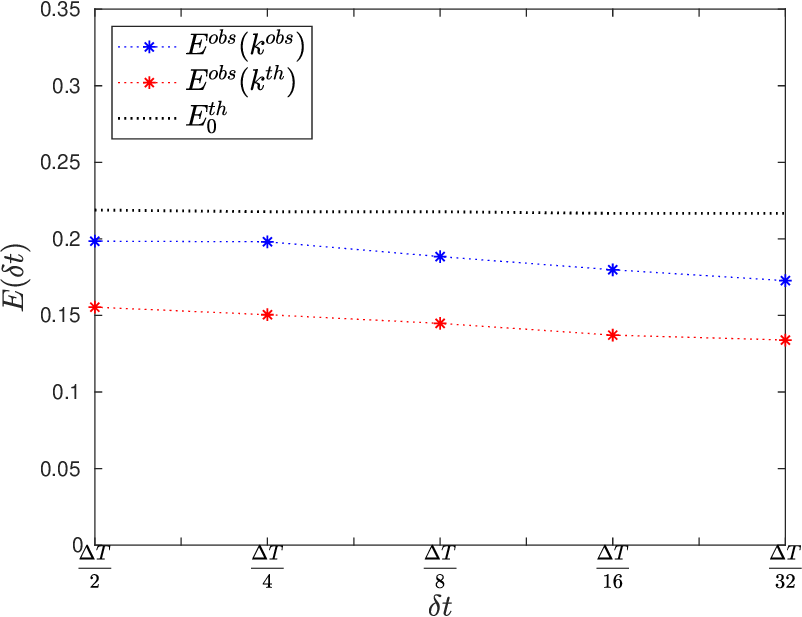}
	\hfill
	\includegraphics[height=5.5cm,width=0.475\linewidth]{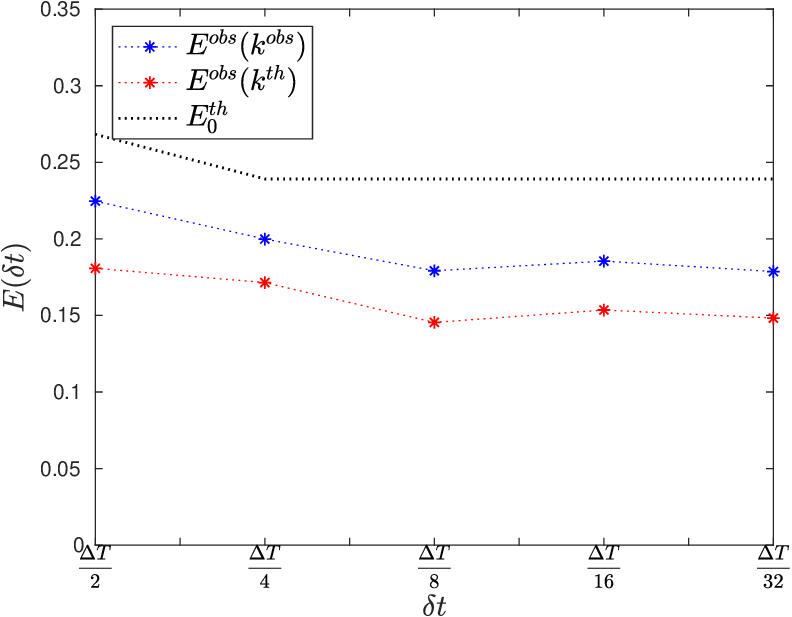}
  	\caption{$E(\delta t)$, for $N = 16$ and $\widetilde{\gamma} = 2^{10}$.}
  	\label{exp2-3}
	\end{subfigure}
\caption[Comparison between $E^{obs}(k^{obs})$, $E^{obs}(k^{th})$ and $B_0^{th}$]{Comparison between $E^{obs}(k^{obs})$, $E^{obs}(k^{th})$ and $E_0^{th}$. The eigenvalues of $A-LC$ are $\{-0.25,-0.5\}$ (left) and $\{-2,-4\}$ (right).}
\label{Exp2}
\end{figure}

\subsection{Variable window approach} 
In the following, we set $M=100$, $\Delta T = 1/16$ and $\{-0.8,-1\}$ as eigenvalues of $A-LC$. 

The Diamond strategy and the Variable window approach are different in nature, but we can compare them by considering the number of parareal iterations as a function of time.
Denoting by $k_{vw}$ the sequence of parareal iterations associated with the latter, we observe in Figure \ref{VW-exp1} that starts performing better than $k^{th}$
, but in the long term underperforms the \textit{a priori} estimate.

Increasing $\widetilde{\gamma}$ leads to a slightly better performance of the Variable window approach, but the behavior previously described still remains. As a consequence of this, the observed efficiency of this procedure is smaller than $E^{obs}(k^{th})$, as shown in Figure \ref{VW-exp2-1}. 

In contrast to the previous subsection, when the observed efficiency depends on $\delta t$, we observe in Figure \ref{VW-exp2-2} a ``jump'' instead of a linear behaviour, due to a decrease in the total number of windows. This can be explained as a ``blindness'' to the tolerance: although the parallelized observer can be closer to the real solution at the end of a window, the Variable window approach does not take this into account and constructs the next one with more than enough subintervals. In principle, the Diamond Strategy faces the same problem, but it is solved using small windows.

\begin{figure}
	\begin{subfigure}{0.475\textwidth}
	\includegraphics[height=4.5cm, width=\linewidth]{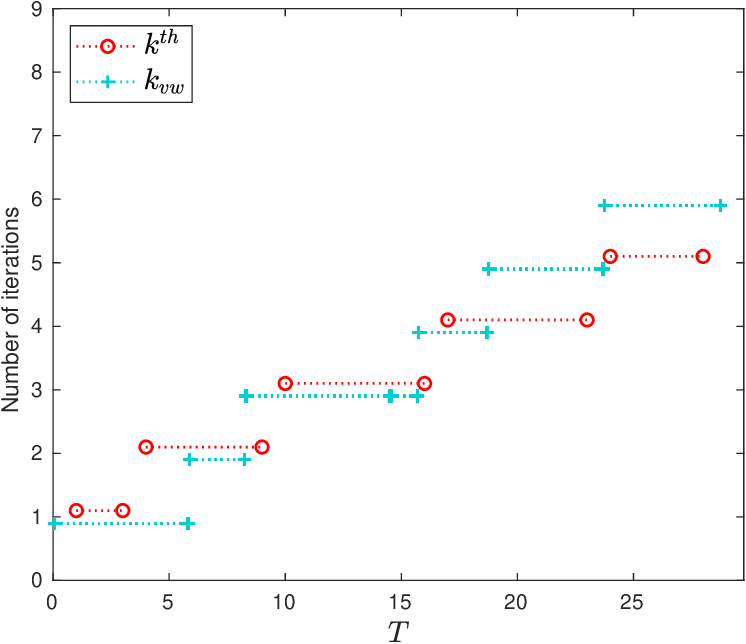}	
	\caption{$\gamma = 1$}
	\label{VW-exp1-1}
	\end{subfigure}
	\hfill
	\begin{subfigure}{0.475\textwidth}
	\adjincludegraphics[height=4.5cm, width=\linewidth, trim={ {.025\width} 0  {.075\width} {.05\width}}, clip,]{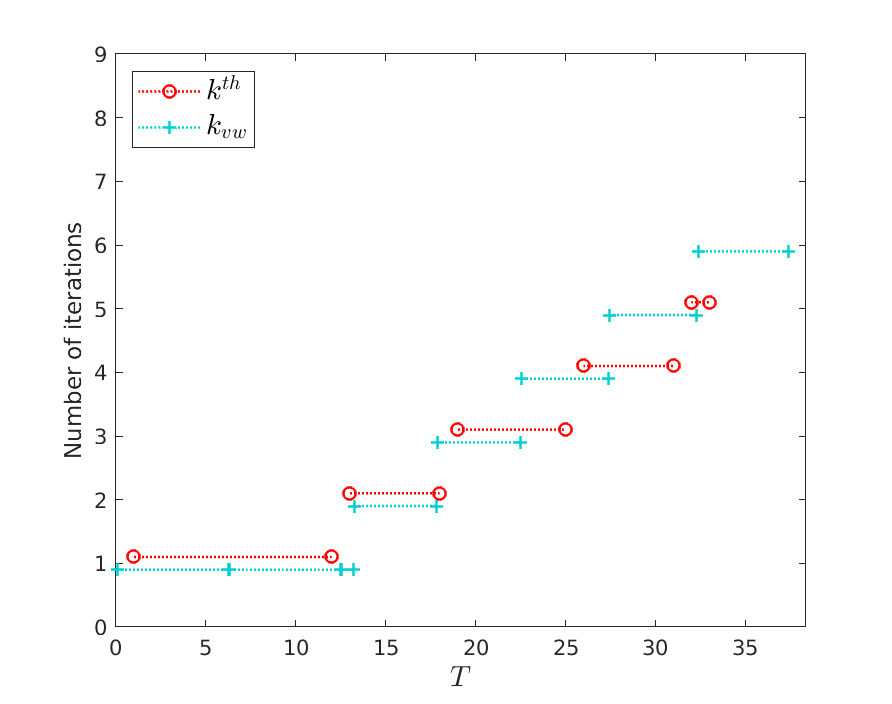}
	\caption{$\gamma = 10^3$}
	\label{VW-exp1-2}	
	\end{subfigure}
\caption{Assimilation for the Diamond strategy and Variable Window approach.}
\label{VW-exp1}
\end{figure}

\begin{figure}
	\begin{subfigure}{0.475\textwidth}
	\includegraphics[height=4.5cm,width=\linewidth]{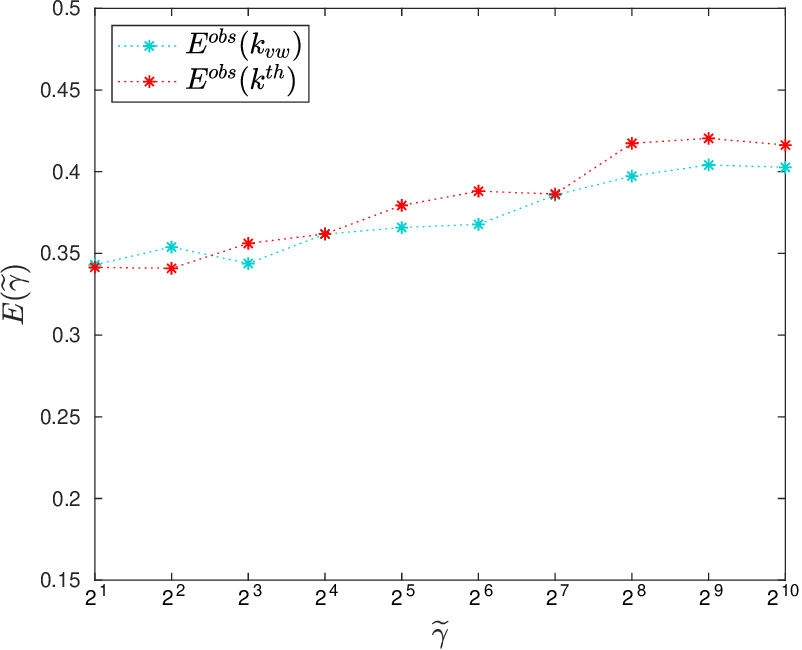}
	\caption{$E(\widetilde{\gamma})$, for $\delta T = \dfrac{\Delta T}{2^5}$.}
  	\label{VW-exp2-1}
	\end{subfigure}
	\hfill
	\begin{subfigure}{0.475\textwidth}
	\includegraphics[height=4.5cm,width=\linewidth]{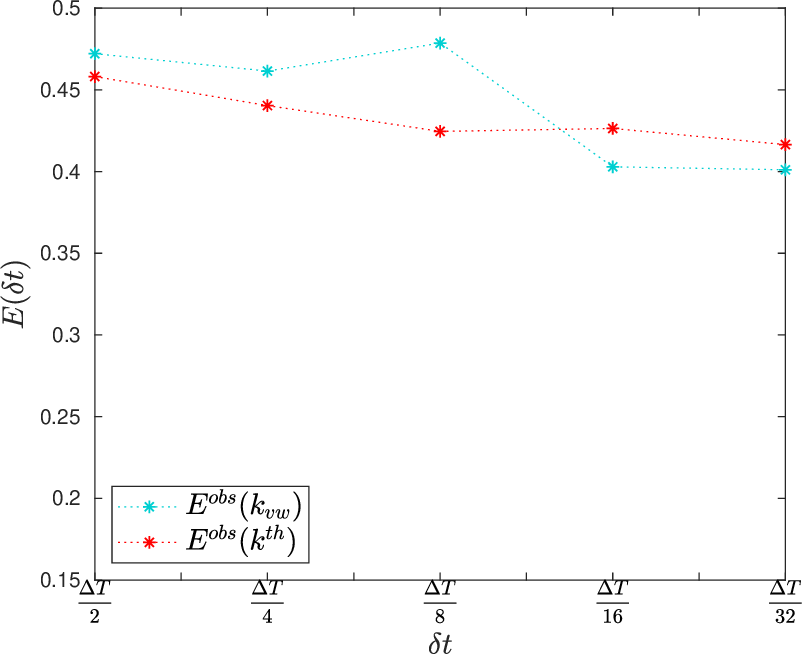}
	\caption{$E(\delta t)$, for $\widetilde{\gamma} = 2^{10}$.}
  	\label{VW-exp2-2}
	\end{subfigure}
\caption{Comparison between the efficiencies $E^{obs}(k^{th})$ (Diamond strategy) and $E^{obs}(k_{vw})$ (Variable Window approach).}
\label{VW-exp2}	
\end{figure}



\section*{Acknowledgments}
This work was supported by ANR Cin\'
e-Para (ANR-15-CE23-0019) and ANR/RGC
ALLOWAP (ANR-19-CE46-0013/A-HKBU203/19), by Swiss National Science Foundation grant
200020 178752, and by Hong Kong Research Grants Council (ECS 22300115 and GRF 12301817).


\bibliographystyle{abbrv}			
\bibliography{parareal_bib}


\appendix
\section{Proof of Proposition \ref{BWDEuler}}\label{appA}
Let $\{t_n\}_{n=0}^N$ be a regular partition of the interval $[0,T]$, with $\Delta T = T/N$. Given $\hat{z}_{n-1}$ an approximation of $\hat{z}(t_{n-1})$, we recall that the Backward Euler method applied to \eqref{Luenberger-diag} is given by 
\begin{equation*}
	\dfrac{\hat{z}_{n} -\hat{z}_{n-1}}{\Delta T} = f(\hat{z}_{n},t_{n}) 
\end{equation*}
where $f(s,t) = Ds +V^{-1}(Bv(t) +Ly(t))$. \\

Since $\delta t$ is assumed to be constant, we then define $\F$ by 
\begin{equation*}
	\F(t_{n},t_{n-1},\hat{z}_{n-1}) 
	= (I -\delta t D)^{-\sfrac{\Delta T}{\delta t}} \left[\delta t V^{-1}(Bv(t_{n-1}) +Lz(t_{n-1})) +\hat{z}_{n-1}\right]
\end{equation*}
and then, a direct calculation yields 
\begin{equation}\label{eq:F}
	\F(t_{n},t_{n-1},w) -\F(t_{n},t_{n-1},z) = (I -\delta t D)^{-\sfrac{\Delta T}{\delta t}}(w-z). 
\end{equation}
On the other hand, $\G$ is defined as a one-step solver, which allows us to replace $\delta t$ by $\Delta T$ in the previous expressions to obtain
\begin{equation}\label{eq:G}
	\G(t_{n},t_{n-1},y) -\G(t_{n},t_{n-1},z) = (I -\Delta t D)^{-1}(y-z). 
\end{equation}
Hence, Definitions \eqref{BE-beta} and \eqref{BE-epsilon} of $\beta$ and $\eta$ follow from combining \eqref{eq:F} and \eqref{eq:G}.\\

To bound the local truncation error, we proceed as follows. Starting at the exact solution $z_{n-1} = \hat{z}(t_{n-1})$, we
define $z_n = \G(t_{n},t_{n-1},z_{n-1})$ and then
\begin{align*}
\tau(t_{n},z_{n-1}) 
	=&  \F(t_{n},t_{n-1},z_{n-1}) - \G(t_{n},t_{n-1},z_{n-1}) \\
	=&  \hat{z}(t_{n}) -z_{n}
\end{align*}
since $\F$ is an exact solver. We use that $z_{n} = \hat{z}(t_{n-1}) +\Delta T f(z_{n},t_{n})$ and then expand $\hat{z}(t_{n-1})$ around $t_{n}$ to get
\begin{equation}\label{LTE}
\tau(t_{n},z_{n-1}) 
	= \Delta T\left(\dot{\hat{z}}(\hat{z}(t_{n}),t_{n})-f(z_{n},t_{n})\right) -\dfrac{(\Delta T)^2}{2}\ddot{\hat{z}}(\hat{z}(\xi),\xi)
\end{equation}
where $\xi\in(t_{n-1},t_{n})$. Since $\dot{\hat{z}} = f(z,t)$, we can get rid of the derivatives of $z$. In particular, the definition of $f(s,t)$ shows that 
\begin{align*}
\dot{\hat{z}}(\hat{z}(t_{n}),t_{n})-f(z_{n},t_{n})
	=&  f(\hat{z}(t_{n}),t_{n}) -f(z_{n},t_{n})
	= D \tau(t_n,z_{n-1}), \\
\ddot{\hat{z}}(\hat{z}(\xi),\xi)
	=&  \dfrac{df}{dt}(\hat{z}(\xi),\xi) 
	= \dfrac{\partial f}{\partial s}(\hat{z}(\xi),\xi) \cdot f(\hat{z}(\xi),\xi) + \dfrac{\partial f}{\partial t}(\hat{z}(\xi),\xi) \\
	=&  D f(\hat{z}(\xi),\xi) + V^{-1}(B\dot{v}(\xi) +L\dot{y}(\xi)). 
\end{align*}
Replacing these expressions in \eqref{LTE} and rearranging terms yields
\[
\tau(t_{n},z_{n-1}) 
	= -\dfrac{(\Delta T)^2}{2}(I-\Delta T D)^{-1}\big[D f(\hat{x}(\xi),\xi) + V^{-1}(B\dot{v}(\xi) +L\dot{y}(\xi))\Big].
\]
Finally, assuming that $\Delta TK < 1$, we take norms and use the definitions of $K$ and $M$ to obtain $\alpha$.

\end{document}